\documentclass[11pt]{article}

\usepackage{mathrsfs}
\usepackage{enumerate}
\usepackage[section]{algorithm}
\usepackage{algorithmic}
\usepackage{multirow}
\usepackage{xcolor}
\usepackage{cancel}

\usepackage{graphicx}
\usepackage{amsmath,amsfonts,amsthm,amsbsy,amssymb}
\usepackage{fixmath}
\usepackage{amssymb,latexsym}

\usepackage{hyperref}

\usepackage{cleveref}

\def\bb{\pmb{b}}

\def\be{\pmb{e}}

\def\bw{\pmb{w}}
\def\bx{\pmb{x}}
\def\by{\pmb{y}}

\def\bone{\pmb{1}}

\def\bbO{\mathbb{O}}
\def\bbP{\mathbb{P}}
\def\bbR{\mathbb{R}}

\def\scrH{\mathscr{H}}

\def\scrR{\mathscr{R}}

\def\cR{{\cal R}}

\def\wtd{\widetilde}
\def\what{\widehat}

\usepackage{accents}

\DeclareMathOperator{\diag}{diag}

\DeclareMathOperator*{\opt}{opt}
\DeclareMathOperator{\rank}{rank}
\DeclareMathOperator{\rmd}{d}

\DeclareMathOperator{\tr}{tr}

\DeclareMathOperator{\F}{F}

\DeclareMathOperator{\T}{T}

\DeclareMathOperator{\KKT}{KKT}

\def\scrR{\mathscr{R}}

\newtheorem{theorem}{Theorem}[section]

\theoremstyle{definition}

\newtheorem{remark}{Remark}[section]

\allowdisplaybreaks

\numberwithin{equation}{section}
\numberwithin{figure}{section}
\numberwithin{table}{section}

\title{An NEPv Approach for Feature Selection \\
via Orthogonal OCCA
 with  the $(2,1)$-norm Regularization}

\author{Li Wang%
\thanks{Department of Mathematics, University of Texas at Arlington, Arlington, TX 76019-0408, USA.
        Supported in part by NSF DMS-2407692.
        Email: {\tt li.wang.edu}.}
\and Lei-Hong Zhang%
\thanks{School of Mathematical Sciences, Soochow University, Suzhou 215006, Jiangsu, China.
 Supported in part by the National Natural Science Foundation of China (NSFC-12471356, NSFC-12371380), Jiangsu Shuangchuang Project (JSSCTD202209),  Academic Degree and Postgraduate Education Reform Project of Jiangsu Province, and China Association of Higher Education under grant 23SX0403.
        Email: {\tt longzlh@suda.edu.cn}.}
\and Ren-Cang Li%
\thanks{Department of Mathematics, University of Texas at Arlington, Arlington, TX 76019-0408, USA.
        Supported in part by NSF DMS-2407692.
        Email: {\tt rcli@uta.edu}.}
}

\date{February 22, 2025
}

\begin{document}

\maketitle

\begin{abstract}
A novel feature selection model via orthogonal canonical correlation analysis
with the $(2,1)$-norm regularization is proposed, and the model  is solved by
a practical NEPv approach
(nonlinear eigenvalue problem with eigenvector dependency),
yielding a feature selection method named OCCA-FS. It is proved that OCCA-FS always produces a sequence of approximations
with monotonic objective values and is globally convergent.
Extensive numerical experiments are performed to compare OCCA-FS against existing feature selection methods.
The numerical results demonstrate that
OCCA-FS produces superior classification performance
and often comes out on the top among all feature selection methods in comparison.


\end{abstract}


\section{Introduction}\label{sec:intro}

High-dimensional datasets have been increasingly available due to the rapid development of information technology. Feature selection is one of the most important tools in machine learning. It aims to select a subset of most important features from a large number of original features for a compact and informative representation \cite{Guel:2003}. As a result, feature selection is generally used as a preprocessing step to remove irrelevant or noisy features and, at the same time,
select the most important features, from  real-world data before conducting other machine learning tasks, such as classification.

Supervised feature selection (SFS) is a type of a feature selection paradigm where the classification label information is available to guide  selection. Methods for SFS generally fall into three categories: filter, wrapper, and embedded \cite{Saeys2007}. Filter methods, such as T-test \cite{Montgomery2007}, rank input features by their intrinsic and statistical characteristics, independent of learning algorithms; Wrapper methods first propose feature subsets and then evaluate them according to certain learning objective, such as SVM in \cite{Guyon2002}; Embedded methods aim to search for an optimal subset of features by solving some specific mathematical learning model, such as  correntropy induced robust feature selection (CRFS) \cite{He2012}.

Least square regression (LSR) is the most common embedded statistical methodology for supervised feature selection. Within the umbrella, there are variations designed for various purposes. In particular, LSR confined to the Stiefel manifold can retain more statistical and structural information \cite{Zhang2018,Wu2020} than otherwise.
For classification, orthogonal least squares regression (OLSR)  can preserve more discriminative properties in the projection subspace than the standard LSR and can avoid trivial solutions \cite{Cai2006}.

Let $X \in \bbR^{n \times p}$ be an input data matrix
whose columns are $p$ sample data points
and let $\by \in \{1, \ldots, k\}^p$ be the associated class labels,
where $n$ is the number of input features, $k$ is the number of class labels.
OLSR aims to estimate a linear function
$$
f(\bx) = P^{\T} \bx + \bb
$$
by minimizing the least square error on the given sample data with respect to the one-hot representation of the class labels
$\by$, where $P \in \bbR^{n \times k}$ is the linear coefficient matrix
restricted to the Stiefel manifold, and $\bb \in \bbR^k$ is the intercept term.
Specifically, OLSR determines the linear function $f(\cdot)$ and the intercept term $\bb$
by the following optimization problem
\begin{equation} \label{eq:opt-OLSR}
	\min_{P \in \bbO^{n\times k},\, \bb \in \bbR^k} \| (P^{\T} X + \bb \bone_p^{\T}) - Y\|_{\F}^2,
\end{equation}
where $Y \in \{0, 1\}^{k \times p}$ is the one-hot representation matrix of  class labels $\by$ of sample data points,
$\bone_p\in\bbR^p$ is the vector of all ones, and
        $$
        \bbO^{n\times k}=\{P\in\bbR^{n\times k}\,:\,P^{\T}P=I_k\}\subset\bbR^{n\times k}
        $$
is the Stielfel manifold. As is commonly done, the intercept term $\bb$ can be eliminated by first minimizing
the objective of \eqref{eq:opt-OLSR} over $\bb$, given $P$, to get
$\bb=(P^{\T}X-Y)\cdot(1/p)\bone_p$. Hence OLSR \eqref{eq:opt-OLSR} is reduced to
\begin{equation} \label{eq:opt-OLSR-no-intercept}
	\min_{P \in \bbO^{n\times k}} \| P^{\T} XC_p - YC_p\|_{\F}^2,
\end{equation}
which will still be referred to  as OLSR henceforth,
where $C_p = I_p - \frac{1}{p}\bone_p \bone_p^{\T}\in\bbR^{p\times p}$ is the centering matrix.
OLSR (\ref{eq:opt-OLSR-no-intercept}) is also known as an
{\em unbalanced orthogonal Procrustes problem\/},
and difficult to solve~\cite{Elpa:1999}.
Expand the objective of OLSR \eqref{eq:opt-OLSR-no-intercept} in terms of matrix traces to get
$$
\| P^{\T} XC_p - YC_p\|_{\F}^2
   =\tr(P^{\T}AP)-2\tr(P^{\T}D)+\tr(YC_pY^{\T}),
$$
where
\begin{equation}\label{eq:A-and-D}
A=XC_pX^{\T}, \quad D=XC_pY^{\T}.
\end{equation}
Since $\tr(YC_pY^{\T})$ is a constant, i.e., independent of $P$, OLSR \eqref{eq:opt-OLSR-no-intercept} is thus
equivalent to
\begin{equation} \label{eq:opt-OLSR-no-intercept'}
	\min_{P \in \bbO^{n\times k}} \big[\tr(P^{\T}AP)-2\tr(P^{\T}D)\big].
\end{equation}
The mathematical idea behind OLSR \eqref{eq:opt-OLSR-no-intercept'} is to bring
$P^{\T} XC_p$ towards $YC_p$ as close as possible in the metric of the matrix Frobenius norm. An alternative idea for the same purpose is via maximizing some kind of ``correlation'' between $P^{\T} XC_p$ and $YC_p$.
Borrowing the idea from canonical correlation analysis (CCA) in statistics while taking $YC_p$ in its entirety, i.e.,
without being projected,
we naturally define the ``correlation'' as
$\tr^2([P^{\T} XC_p][YC_p]^{\T})/\tr([P^{\T} XC_p][P^{\T} XC_p]^{\T})$, leading to the following orthogonal-type CCA formulation, instead of OLSR \eqref{eq:opt-OLSR-no-intercept'},
\begin{equation} \label{eq:opt-OCCA}
\max_{P \in \bbO^{n\times k}} \frac {\tr^2(P^{\T}D)}{\tr(P^{\T}AP)},
\end{equation}
where $A$ and $D$ are as defined in \eqref{eq:A-and-D}.
Optimization problem in this form appeared previously  in \cite{zhwb:2022} and was called
the OCCA (orthogonal CCA) subproblem. Henceforth for convenience, it will be abusively referred to as OCCA for short.


For feature selection, the $(2,1)$-norm regularization is often employed to induce row sparsity in the linear coefficient matrix $P$ in order to rank features \cite{He2012,Zhang2018}. OCCA \eqref{eq:opt-OCCA}
with  the $(2,1)$-norm regularization takes the form
\begin{equation} \label{eq:opt-OCCA+L21}
\max_{P\in\bbO^{n\times k}}\left\{f(P):=\frac {[\tr(P^{\T}D)]^2}{\tr(P^{\T}AP)}-\alpha\|P\|_{2,1}\right\},
\end{equation}
which will be referred to as OCCA21 henceforth,
where $\alpha>0$ is the regularization parameter, and the $(2,1)$-norm $\|P\|_{2,1}$ is defined
as
$$
\|P\|_{2,1}=\sum_{i=1}^n\|P_{(i,:)}\|_2
$$
with $\|P_{(i,:)}\|_2$ denoting the vector $2$-norm of the $i$th row vector of $P$.
With a properly chosen regularization parameter $\alpha$, it is expected that the optimizer of
OCCA21 \eqref{eq:opt-OCCA+L21} have (many)
negligible rows, i.e., rows such that $\|P_{(i,:)}\|_2\ll 1$. Features corresponding to those negligible rows
in the sample data points are then considered not important and thus can be deselected,
while the rest of the features are selected. In practice, however, this is often
performed by sorting all row norms $\|P_{(i,:)}\|_2$ decreasingly and then
selecting those features corresponding to first many largest row norms.

The rest of this paper is organized as follows.
In \cref{sec:OCCAvsOLSR}, we explore a connection between OCCA21 \eqref{eq:opt-OCCA+L21} and
a recent OLSR-OS (OLSR with optimal scaling) model of \cite{Zhang2018} so as to explain that
\eqref{eq:opt-OCCA+L21} implicitly conceals the optimal scaling in OLSR-OS.
In \cref{sec:OCCA+L21},
we
establish our NEPv theory on OCCA21 \eqref{eq:opt-OCCA+L21} and our practical NEPv
approach for numerically solving \eqref{eq:opt-OCCA+L21}, yielding our own feature selection method
which we will call OCCA-FS.
Extensive numerical experiments on  OCCA-FS in comparison with three
other the-state-of-art methods are conducted in \cref{sec:egs}. Finally conclusions are drawn in \cref{sec:concl}.

{\bf Notation.}
  $\bbR^{m\times n}$  is the set of $m\times n$ real matrices,  $\bbR^n=\bbR^{n\times 1}$, and $\bbR=\bbR^1$;
        $\bbR_+:=\{x\in\bbR\,:\,x\ge 0\}$ and $\bbR_{++}:=\{x\in\bbR_+\,:\,x> 0\}$.
  $I_n\in\bbR^{n\times n}$ is the identity matrix or simply $I$ if its size is clear from the context, and $\be_j$ is the $j$th column of $I$ of an apt size.

  For $B\in\bbR^{m\times n}$, $B^{\T}$ stands for the transpose of a matrix/vector $B$;
      $\rank(B)$ is the rank of $B$, and
      $B_{(i,:)}$ is the $i$th row of $B$. For $A\in\bbR^{n\times n}$, $A\succ 0\, (\succeq 0)$ means that it is symmetric and positive definite (semi-definite), and
        accordingly
        $A\prec 0\, (\preceq 0)$ if $-A\succ 0\, (\succeq 0)$.

  The {\em thin\/} SVD of $B\in\bbR^{m\times n}$ ($m\ge n$) refers to $B=U\Sigma V^{\T}$ with
        $$
        \Sigma=\diag(\sigma_1(B),\sigma_2(B),\ldots,\sigma_n(B))\in\bbR^{n\times n},
        \,\,
        U\in\bbO^{m\times n},\,\,
        V\in\bbO^{n\times n}.
        $$
        The singular values $\sigma_j(B)$  are always arranged decreasingly as
        $
        \sigma_1(B)\ge\cdots\ge\sigma_n(B)\ge 0.
        $
        Accordingly, $\|B\|_2$, $\|B\|_{\F}$, and $\|B\|_{\tr}$ are the spectral, Frobenius, and
        trace         norms of $B$:
        $$
        \|B\|_2=\sigma_1(B),\,\,
        \|B\|_{\F}=\Big(\sum_{i=1}^n[\sigma_i(B)]^2\Big)^{1/2},\,\,
        \|B\|_{\tr}=\sum_{i=1}^n\sigma_i(B),
        $$
        respectively. The trace norm is  also known as the {\em nuclear norm}. Unless otherwise explicitly stated, SVD always
        refers to {\em thin\/} SVD.
        With the SVD, a polar decomposition of $B$ is given by $B=QH$ \cite{high:2008} where $Q=UV^{\T}$
        and $H=V^{\T}\Sigma V$,  and
        $Q$ is called  an {\em orthogonal polar factor\/} and it is unique if $\rank(B)=n$ \cite{li:1993b,li:2014HLA}.

\section{OCCA vs. OLSR}\label{sec:OCCAvsOLSR}
\subsection{Connection to OLSR-OS}
Recently, the authors in \cite{Zhang2018} noticed that there could be a mismatch in magnitude
between sample data in $X$ and labels in $Y$ which is usually of $O(1)$. Confined to the Stiefel manifold,
$P$ cannot compensate large discrepancy in magnitude, if any, between the two. To overcome such a potential discrepancy, they introduced a scaling parameter to better
align the magnitudes of sample data points with their labels and called the new model OLSR-OS.
Specifically, OLSR-OS  determines the linear function $f(\bx) = P^{\T} \bx + \bb$
by
\begin{equation} \label{eq:opt-OLSR-OS}
	\min_{\gamma\in\bbR,\, P \in \bbO^{n\times k},\, \bb \in \bbR^k} \| \gamma(P^{\T} X + \bb \bone_p^{\T}) - Y\|_{\F}^2,
\end{equation}
where $\gamma$ is a scaling parameter. Similarly, the intercept term $\bb$ can be optimized out first to yield
\begin{equation} \label{eq:opt-OLSR-OS-no-intercept}
	\min_{\gamma\in\bbR,\,P \in \bbO^{n\times k}} \| P^{\T}(\gamma X)C_p - YC_p\|_{\F}^2.
\end{equation}
Expand the objective of OLSR-OS \eqref{eq:opt-OLSR-OS-no-intercept} in terms of matrix traces to get
$$
\| P^{\T}(\gamma X)C_p - YC_p\|_{\F}^2
   =\gamma^2\tr(P^{\T}AP)-2\gamma\tr(P^{\T}D)+\tr(YC_pY^{\T}),
$$
where $A$ and $D$ are as in \eqref{eq:A-and-D}. Hence OLSR-OS \eqref{eq:opt-OLSR-OS-no-intercept} is
equivalent to
\begin{equation}\label{eq:opt-OLSR-OS'}
\min_{P\in\bbO^{n\times k},\,\gamma\in\bbR}
   \Big\{g(P,\gamma):=\gamma^2\tr(P^{\T}AP)-2\gamma\tr(P^{\T}D)\Big\}.
\end{equation}
Evidentally $A\succeq 0$ in \eqref{eq:A-and-D}, but we will further assume $\rank(A)>n-k$, which is needed to ensure $\tr(P^{\T}AP)>0$ for any $P\in\bbO^{n\times k}$ and hence
$g(P,\gamma)$ is always bounded from below for $\gamma\in\bbR$.
Given $P$, $g(P,\gamma)$ over $\gamma\in\bbR$ achieves its minimum at
\begin{equation}\label{eq:P2gamma}
\gamma=\frac {\tr(P^{\T}D)}{\tr(P^{\T}AP)},
\end{equation}
substituting which into \eqref{eq:opt-OLSR-OS'}, upon considering maximizing $-g(P,\gamma)$ instead, turns \eqref{eq:opt-OLSR-OS'} equivalently into OCCA \eqref{eq:opt-OCCA}.
In other words, OCCA \eqref{eq:opt-OCCA} is OLSR-OS \eqref{eq:opt-OLSR-OS'} with scaling parameter $\gamma$ optimized out.

It is noted that any discrepancy in magnitude between $X$ and $Y$ does not influence the optimizers to
OCCA \eqref{eq:opt-OCCA} at all, which is not surprising because the above connection simply implies
 there is some optimal scaling inherently.

Upon introducing the $(2,1)$-norm regularization, we end up with OCCA21
\eqref{eq:opt-OCCA+L21}, the model of OCCA regularized by the $(2,1)$-norm.
In \cref{sec:OCCA+L21}, we will focus on the NEPv theory for OCCA21 \eqref{eq:opt-OCCA+L21} and how to numerically solve it by a
practical NEPv approach.

\subsection{Solving regularized OLSR-OS}\label{ssec:mistake-Zhang2018}
In \cite{Zhang2018}, the authors  combined OLSR-OS \eqref{eq:opt-OLSR-OS'} with
the $(2,1)$-norm regularization
\begin{equation} \label{eq:opt-OLSR-OS21}
\min_{P\in\bbO^{n\times k},\,\gamma\in\bbR}\Big\{g(P,\gamma):=\gamma^2\tr(P^{\T}AP)-2\gamma\tr(P^{\T}D)+\alpha\|P\|_{2,1}\Big\},
\end{equation}
which will be referred to as OLSR-OS21 henceforth.
Also in \cite{Zhang2018}, an alternating optimization procedure
between $\gamma$ and $P$ was designed to solve \eqref{eq:opt-OLSR-OS21}, yielding an eventual supervised
feature selection method named PEB-FS.


Unfortunately, the alternating optimization procedure is not properly done. We shall now explain.
PEB-FS  \cite{Zhang2018} solves (\ref{eq:opt-OLSR-OS21})
alternatingly
between $\gamma$ and $P$ with
\begin{align*}
    \mbox{the $\gamma$-update:} \,\,&\mbox{minimize $g(P,\gamma)$ over $\gamma$ given $P$, and} \\
    \mbox{the $P$-update:} \,\,&\mbox{minimize $g(P,\gamma)$ over $P$ given $\gamma$.}
\end{align*}
The $\gamma$-update is rather easy and has a close-form solution in \eqref{eq:P2gamma}, but the $P$-update is not and that is
where the mistake was committed. Specifically, to find the $P$-update, in \cite{Zhang2018},
the authors rewrite $g(P,\gamma)$ as follows. Let
\begin{subequations}\label{eq:g(P,gamma)-rewrite}
\begin{equation}\label{eq:g(P,gamma)-rewrite:1}
\Gamma(P)=\diag(\|P_{(1,:)}\|_2^{-1},\|P_{(2,:)}\|_2^{-1},\ldots,\|P_{(n,:)}\|_2^{-1}),
\end{equation}
and let $P_{\bot}\in\bbO^{n\times (n-k)}$ such that $Q:=[P,P_{\bot}]\in\bbO^{n\times n}$. Then we have
\begin{align}
\|P\|_{2,1}&=\tr(\Gamma(P)PP^{\T})=\tr(P^{\T}\Gamma(P)P),\label{eq:g(P,gamma)-rewrite:2} \\
g(P,\gamma)&=\tr(Q^{\T}[\gamma^2A+\alpha\Gamma(P)]Q)-2\gamma\tr(Q^{\T}R(P)),
     \label{eq:g(P,gamma)-rewrite:3}
\end{align}
where $R(P):=[D,(\gamma/2)AP_{\bot}+(\alpha/(2\gamma))\Gamma(P)P_{\bot}]\in\bbR^{n\times n}$.
Now given $\gamma$, if $\Gamma(P)$ were constant (i.e., independent of $P$), then $\tr(Q^{\T}[\gamma^2A+\alpha\Gamma(P)]Q)$
would be constant too,  and thus minimizing $g(P,\gamma)$ over $P$ would be
equivalent to maximizing $\tr(Q^{\T}R(P))$, which could be done by the SVD of $R(P)$ if $R(P)$
were also constant. Unfortunately,
both $\Gamma(P)$ and $R(P)$ are in general dependent of $P$, and hence minimizing $g(P,\gamma)$ over $P$,
given $\gamma$, cannot be simply done by the SVD of $R(P)$ evaluated at the current approximation,
but that is exactly what
the $P$-update was carried out in \cite{Zhang2018}. For that reason, the conclusion of \cite[Theorem 3.2]{Zhang2018}
that PEB-FS generates monotonically decreasing objective values is not correct. In fact, later in
\Cref{fig:convergence} it is observed that objective values by PEB-FS is not monotonically decreasing
on dataset Yale. Also in the figure, we see that the optimal objective values by OCCA-FS is far smaller than those delivered by PEB-FS
for all considered datasets.
\end{subequations}

It is noted that the $P$-update can be handled by the NEPv approach in a similar way to what we will do to
OCCA21 \eqref{eq:opt-OCCA+L21} in the next section. For that, we will simply make a remark later in \Cref{rk:PEB-FS}
but will not pursue it in this paper because OCCA21 \eqref{eq:opt-OCCA+L21} can merge both updates into one, kind of
``kill two birds with one stone!''

%
%

\section{NEPv for OCCA with the (2,1)-norm regularization}\label{sec:OCCA+L21}

Earlier, we mentioned that the role of $\alpha$ in OCCA21~\eqref{eq:opt-OCCA+L21}, and in
OLSR-OS21  \eqref{eq:opt-OLSR-OS21} for that matter, is to drive
some rows of the respective optimizers towards $0$. But owing to the fact that $P\in\bbO^{n\times k}$, $P$ has at least $k$ rows that have
nontrivial norms, i.e., not small, as guaranteed by \Cref{thm:not-small} below.

\begin{theorem}\label{thm:not-small}
Let $P\in\bbO^{n\times k}$ and rearrange $\{\|P_{(i,:)}\|_2\}_{i=1}^n$ descendingly as
$$
\|P_{(i_1,:)}\|_2\ge\|P_{(i_2,:)}\|_2\ge\cdots\ge\|P_{(i_n,:)}\|_2.
$$
Then
$$
\|P_{(i_j,:)}\|_2\ge\sqrt{\frac {k-j+1}{n-j+1}}\ge\frac 1{\sqrt{n-k+1}}\quad\mbox{for $1\le j\le k$}.
$$
\end{theorem}

\begin{proof}
For $1\le j\le k$, noticing $\|P_{(i,:)}\|_2\le 1$ for any $i$, we have
$$
k=\sum_{j=1}^n\|P_{(i_j,:)}\|_2^2
 \le (j-1)+(n-j+1)\|P_{(i_j,:)}\|_2^2,
$$
implying
$$
\|P_{(i_j,:)}\|_2^2\ge\frac {k-j+1}{n-j+1}\ge\frac 1{n-k+1},
$$
as was to be shown.
\end{proof}

\subsection{The NEPv Approach}\label{ssec:thy}
We rewrite $\|P\|_{2,1}$ in terms of matrix traces as follows:
\begin{equation}\label{eq:L21->trace}
\|P\|_{2,1}=\sum_{i=1}^n\|\be_i^{\T}P\|_2
    =\sum_{i=1}^n\sqrt{\be_i^{\T}PP^{\T}\be_i}
    =\sum_{i=1}^n\sqrt{\tr(P^{\T}\be_i\be_i^{\T}P)}\,,
\end{equation}
and, hence, for OCCA21 \eqref{eq:opt-OCCA+L21},
\begin{subequations}\label{eq:opt-OLSR+L21:-obj}
\begin{align}
f(P)&=\frac {[\tr(P^{\T}D)]^2}{\tr(P^{\T}AP)}-\alpha\sum_{i=1}^n\sqrt{\tr(P^{\T}\be_i\be_i^{\T}P)} \label{eq:opt-OLSR+L21:-obj-1} \\
    &=\phi\circ T(P), \label{eq:opt-OLSR+L21:-obj-2}
\end{align}
where, for $\bx\equiv [x_i]\in\bbR^{n+2}$,
\begin{equation}\label{eq:opt-OLSR+L21:-obj-3}
\phi(\bx)=\frac {x_2^2}{x_1}-\alpha\sum_{i=1}^n\sqrt{x_{i+2}}, \quad
T(P)=\begin{bmatrix}
       \tr(P^{\T}AP) \\
       \tr(P^{\T}D) \\
       \tr(P^{\T}\be_1\be_1^{\T}P) \\
       \vdots \\
       \tr(P^{\T}\be_n\be_n^{\T}P)
     \end{bmatrix}.
\end{equation}
\end{subequations}
Function $\phi(\bx)$ is convex in its first two components in $\bbR_{++}\times\bbR$ \cite[p.72]{bova:2004}
and componentwise convex in the rest of its components in $\bbR_+$ because
$$
\frac {\rmd^2 (-\alpha\sqrt{x})}
      {\rmd x^2}=\frac 14\alpha x^{-3/2}\ge 0.
$$
Hence $\phi(\bx)$ is convex for $\bx\in\bbR_{++}\times\bbR\times\bbR_+^n$. In  other words,
$f(P)$ is a convex composition of $n+2$ matrix trace functions:
$\tr(P^{\T}AP)$, $\tr(P^{\T}D)$, $\tr(P^{\T}\be_i\be_i^{\T}P)$ for $1\le i\le n$, all of which are
atomic functions for NEPv \cite[Theorems~7.4 and 7.5]{li:2024}.
Therefore it is natural to ask if the NEPv approach \cite[Part~II]{li:2024} can be utilized
to solve \eqref{eq:opt-OCCA+L21}.

%
%
%

In what follows, we will present the NEPv approach for the current case. Let
\begin{equation}\label{eq:h(P)}
h(P)=\frac {\tr(P^{\T}D)}{\tr(P^{\T}AP)}.
\end{equation}
We have
\begin{subequations}\label{eq:opt-OLSR+L21:KKT'}
\begin{equation}\label{eq:opt-OLSR+L21:scrH}
\scrH(P):=\frac {\partial f(P)}{\partial P}
    =2 h(P) \big[D-h(P)\,AP\big]
     -\alpha\sum_{i=1}^n\frac {\be_i\be_i^{\T}P}{\sqrt{\tr(P^{\T}\be_i\be_i^{\T}P)}},
\end{equation}
unless some $\tr(P^{\T}\be_i\be_i^{\T}P)=\|\be_i^{\T}P\|_2^2=0$ for which case $f(P)$ is not differentiable.
As a result, the first order optimality condition, also known as the KKT condition, is
\cite[section~2]{li:2024}
\begin{equation}\label{eq:opt-OLSR+L21:KKT}
\scrH(P)=P\Lambda, \quad P\in\bbO^{n\times k},\quad \Lambda=\Lambda^{\T}\in\bbR^{k\times k},
\end{equation}
\end{subequations}
which governs any KKT point $P$ that has no zero rows.

To apply the NEPv approach, we will need a symmetric matrix-valued
function that can be constructed from each individual ones for the $n+2$ matrix trace functions:
$\tr(P^{\T}AP)$, $\tr(P^{\T}D)$, $\tr(P^{\T}\be_i\be_i^{\T}P)$ for $1\le i\le n$, as atomic functions for NEPv.
According to \cite[(8.3)]{li:2024}, we can use
\begin{subequations}\label{eq:opt-OLSR+L21:NEPv'}
\begin{equation}\label{eq:opt-OLSR+L21:H(P)}
H(P):=2 h(P) \big[\big(DP^{\T}+PD^{\T})-h(P)\,A\big]
-\alpha\sum_{i=1}^n\frac {\be_i\be_i^{\T}}{\sqrt{\tr(P^{\T}\be_i\be_i^{\T}P)}}
   \in\bbR^{n\times n},
\end{equation}
whose associated NEPv is
\begin{equation}\label{eq:opt-OLSR+L21:NEPv}
H(P)\,P=P\Omega, \quad P\in\bbO^{n\times k},\quad \Omega=\Omega^{\T}\in\bbR^{k\times k}.
\end{equation}
\end{subequations}
We find that for $P\in\bbO^{n\times k}$
\begin{equation}\label{eq:H(p)-scrH(P)}
H(P)P-\scrH(P)=P\big[2h(P)\,D^{\T}P\big].
\end{equation}
The next theorem is a straightforward consequence of \cite[Theorem 6.1]{li:2024} and
\eqref{eq:H(p)-scrH(P)}.

\begin{theorem}[{\cite[Theorem 6.1]{li:2024}}]\label{thm:H(P)-eligibility}
Suppose that \eqref{eq:opt-OLSR-OS21} has no KKT point $P$ with zero rows.
$P\in\bbO^{n\times k}$ is a solution to the KKT condition \eqref{eq:opt-OLSR+L21:KKT'} if and only if
it is a solution to NEPv \eqref{eq:opt-OLSR+L21:NEPv'}
and $D^{\T}P$ is  symmetric.
\end{theorem}

We restate \cite[Theorem 8.1]{li:2024} for the current case as follows.

\begin{theorem}[{\cite[Theorem 8.1]{li:2024}}]\label{thm:main-NEPv-cvx}
Consider \eqref{eq:opt-OCCA+L21} with $f(P)$ as in
\eqref{eq:opt-OLSR+L21:-obj}. Let $H(P)$ be given by \eqref{eq:opt-OLSR+L21:H(P)} and
\begin{equation}\label{eq:bbP}
\bbP:=\{P\in\bbO^{n\times k}\,:\,P^{\T}D\succeq 0\}\subseteq\bbO^{n\times k}.
\end{equation}
For $\what P\in\bbO^{n\times k}$ and $P\in\bbP$, if $P$ has no zero rows and if
\begin{equation}\label{eq:NEPv-assume}
\tr(\what P^{\T}H(P)\what P)\ge\tr(P^{\T}H(P)P)+\eta
\quad\mbox{for some $\eta\in\bbR$},
\end{equation}
then
$f(\wtd P)\ge f(P)+\eta/2$,
where $\wtd P=\what PQ\in\bbP$ with $Q$ being an orthogonal polar factor of $\what P^{\T}D$.
\end{theorem}

We have stated a seemingly beautiful theory, but there is one unsettling condition in both \Cref{thm:H(P)-eligibility,thm:main-NEPv-cvx}: $P$ has no zero rows. The condition has to be there
%
because of the singularity in the expression of $\scrH(P)$ in \eqref{eq:opt-OLSR+L21:scrH} and that of $H(P)$ in \eqref{eq:opt-OLSR+L21:NEPv'} when one or more rows of $P$ are $0$.
In order not to miss out any KKT points, we will have to look for those KKT points with one or more zero rows. This can be done, in theory, by exhausting all possibilities, e.g., looking
for KKT points $P$ with zero rows. The number of all possibilities is $2^{n-k}$ because of \Cref{thm:not-small},  making this option theoretically correct but practically infeasible.


\subsection{A practical NEPv approach}\label{ssec:practical}


The precise purpose of introducing
the (2,1)-norm regularization in \eqref{eq:opt-OCCA+L21} and \eqref{eq:opt-OLSR-OS21}
is to promote row sparsity in $P$, i.e., making some of the rows of $P$ zeros in theory. When that happens,
singularities occur, as they are exposed in \eqref{eq:opt-OLSR+L21:scrH} and \eqref{eq:opt-OLSR+L21:H(P)}.
Numerically, it is unlikely that some rows of $P$ become zeros exactly during any optimization process, but
near singularity is bound to occur and near-singularity can cause numerical ill-behavior.
Fortunately, near-singularity can be effectively monitored by checking if some $\|\be_i^{\T}P\|_2\ll 1$.
When $\|\be_i^{\T}P\|_2\ll 1$, it means that, from the feature selection point of view,
the $i$th feature is insignificant and may be deselected. In practice, this can be realized by pre-selecting a small tolerance
$\varepsilon_0$ and then regarding any rows such that
$\|\be_i^{\T}P\|_2\le\varepsilon_0$ insignificant. How small should $\varepsilon_0$ be? Since
$$
\sum_{i=1}^n\|\be_i^{\T}P\|_2^2=k
\quad\Rightarrow\quad
\frac kn\le\max_{1\le i\le n}\|\be_i^{\T}P\|_2^2\le 1.
$$
For data science applications, $\varepsilon_0=10^{-3}\sqrt{k/n}$ should be good enough.

In view of these discussions, instead of \eqref{eq:opt-OCCA+L21}, we may solve
a perturbed problem of it:
\begin{equation}\label{eq:opt-OLSR+L21:work-eps}
\max_{P\in\bbO^{n\times k}} \left\{ f_{\varepsilon_0}(P):=\frac {[\tr(P^{\T}D)]^2}{\tr(P^{\T}AP)}-\alpha\sum_{i=1}^n\sqrt{\tr(P^{\T}\be_i\be_i^{\T}P)+\varepsilon_0^2}\right\}.
\end{equation}
With the same $T(P)$ as in \eqref{eq:opt-OLSR+L21:-obj-3}, we find that
$$
f_{\varepsilon_0}(P)=\phi_{\varepsilon_0}\circ T(P)
\quad\mbox{with}\quad
\phi_{\varepsilon_0}(\bx)=\frac {x_2^2}{x_1}-\alpha\sum_{i=1}^n\sqrt{x_{i+2}+\varepsilon_0^2},
$$
Function $\phi_{\varepsilon_0}(\bx)$ is still convex for $\bx\in\bbR_{++}\times\bbR\times\bbR_+^n$.
Accordingly, we have
\begin{subequations}\label{eq:opt-OLSR+L21:KKT'-eps}
\begin{equation}\label{eq:opt-OLSR+L21:scrH-eps}
\scrH_{\varepsilon_0}(P):=\frac {\partial f_{\varepsilon_0}(P)}{\partial P}
    =2 h(P) \big[D-h(P)\,AP\big]
     -\alpha\sum_{i=1}^n\frac {\be_i\be_i^{\T}P}{\sqrt{\tr(P^{\T}\be_i\be_i^{\T}P)+\varepsilon_0^2}},
\end{equation}
where $h(P)$ is as defined  in \eqref{eq:h(P)},
and the first order optimality condition of \eqref{eq:opt-OLSR+L21:work-eps} is then given by
\cite[section~2]{li:2024}
\begin{equation}\label{eq:opt-OLSR+L21:KKT-eps}
\scrH_{\varepsilon_0}(P)=P\Lambda, \quad P\in\bbO^{n\times k},\quad \Lambda=\Lambda^{\T}\in\bbR^{k\times k}.
\end{equation}
\end{subequations}
Correspondingly, we use symmetric matrix-valued function
\begin{subequations}\label{eq:opt-OLSR+L21:NEPv'-eps}
\begin{equation}\label{eq:opt-OLSR+L21:H(P)-eps}
H_{\varepsilon_0}(P):=2 h(P) \big[\big(DP^{\T}+PD^{\T})-h(P)\,A\big]
     -\alpha\sum_{i=1}^n\frac {\be_i\be_i^{\T}}{\sqrt{\tr(P^{\T}\be_i\be_i^{\T}P)+\varepsilon_0^2}}
   \in\bbR^{n\times n},
\end{equation}
whose associated NEPv is
\begin{equation}\label{eq:opt-OLSR+L21:NEPv-eps}
H_{\varepsilon_0}(P)\,P=P\Omega, \quad P\in\bbO^{n\times k},\quad \Omega=\Omega^{\T}\in\bbR^{k\times k}.
\end{equation}
\end{subequations}
Again we find that for $P\in\bbO^{n\times k}$
\begin{equation}\label{eq:H(p)-scrH(P)-eps}
H_{\varepsilon_0}(P)P-\scrH_{\varepsilon_0}(P)=P\big[2h(P)\,D^{\T}P\big].
\end{equation}
We will also have the corresponding versions of \Cref{thm:H(P)-eligibility,thm:main-NEPv-cvx} for \eqref{eq:opt-OLSR+L21:work-eps}, without the need to
assume that $P$ has no zero rows. They are stated as follows.

\begin{algorithm}[t]
\caption{The NEPv approach for solving \eqref{eq:opt-OLSR+L21:work-eps}} \label{alg:NEPvSCF4OLSR+L21}
\begin{algorithmic}[1]
\REQUIRE $0\preceq A\in\bbR^{n\times n}$ (such that $\rank(A)>n-k$), $D\in\bbR^{n\times k}$, regularization parameter $\alpha>0$,
         $\varepsilon_0>0$, $H_{\varepsilon_0}(\cdot)$ as \eqref{eq:opt-OLSR+L21:H(P)-eps},
         and initial  $P^{(0)}\in\bbO^{n\times k}$;
\ENSURE  an approximate maximizer of \eqref{eq:opt-OLSR+L21:work-eps}. 
\STATE if $\big[P^{(0)}\big]^{\T}D\not\succeq 0$, update $P^{(0)}$ to $P^{(0)}Q_0$ where
       $Q_0=U_0V_0^{\T}$, an orthogonal polar factor of $\big[P^{(0)}\big]^{\T}D=U_0\Sigma_0 V_0$ (SVD);
\FOR{$j=0,1,\ldots$ until convergence}
    \STATE compute $H_j=H_{\varepsilon_0}(P^{(j)})\in\bbR^{n\times n}$;
    \STATE solve symmetric eigenvalue problem $H_j\what P^{(j)}=\what P^{(j)}\Omega_j$ for $\what P^{(j)}\in\bbR^{n\times k}$,
           an orthonormal basis matrix of the eigenspace associated with the first $k$ largest eigenvalues of $H_j$;
    \STATE compute the SVD $\big[\what P^{(j)}\big]^{\T}D=U_j\Sigma_j V_j$ and let $Q_j=U_jV_j^{\T}$, an orthogonal polar
           factor of $\big[\what P^{(j)}\big]^{\T}D$;
    \STATE let $P^{(j+1)}=\what P^{(j)}Q_j$;
\ENDFOR
\RETURN the last $P^{(j)}$.
\end{algorithmic}
\end{algorithm}

\begin{theorem}[{\cite[Theorem 6.1]{li:2024}}]\label{thm:H(P)-eligibility-eps}
$P\in\bbO^{n\times k}$ is a solution to the KKT condition \eqref{eq:opt-OLSR+L21:KKT'-eps} of \eqref{eq:opt-OLSR+L21:work-eps} if and only if
it is a solution to NEPv \eqref{eq:opt-OLSR+L21:NEPv'-eps}
and $D^{\T}P$ is  symmetric.
\end{theorem}

\begin{theorem}
\label{thm:main-NEPv-cvx-eps}
Let $H_{\varepsilon_0}(P)$ be given by \eqref{eq:opt-OLSR+L21:NEPv'-eps} and
$\bbP$ as defined in \eqref{eq:bbP}.
\begin{enumerate}[{\rm (a)}]
  \item {\rm \cite[Theorem 8.1]{li:2024}} For $\what P\in\bbO^{n\times k}$ and $P\in\bbP$,  if
        \begin{equation}\label{eq:NEPv-assume-eps}
        \tr(\what P^{\T}H_{\varepsilon_0}(P)\what P)\ge\tr(P^{\T}H_{\varepsilon_0}(P)P)+\eta
        \quad\mbox{for some $\eta\in\bbR$},
        \end{equation}
        then
        $f_{\varepsilon_0}(\wtd P)\ge f_{\varepsilon_0}(P)+\eta/2$,
        where $\wtd P=\what PQ\in\bbP$ with $Q$ being an orthogonal polar factor of $\what P^{\T}D$.
  \item Let $P_*\in\bbO^{n\times k}$ be an maximizer of \eqref{eq:opt-OLSR+L21:work-eps}. Then $P_*\in\bbP$,
        and $P_*$ satisfies NEPv \eqref{eq:opt-OLSR+L21:NEPv-eps}, and the eigenvalues of $\Omega_*:=P_*^{\T}H_{\varepsilon_0}(P_*)P_*$ are
        the $k$ largest eigenvalues of $H_{\varepsilon_0}(P_*)$.
\end{enumerate}
\end{theorem}

\begin{proof}
As cited, item (a) is due to \cite[Theorem 8.1]{li:2024}. For item (b), if we can show that $P_*\in\bbP$ then the rest follows from
\cite[Theorem~6.3]{li:2024}. Assume, to the contrary, that $P_*\not\in\bbP$, i.e., $P_*^{\T}D\not\succeq 0$, which
implies \cite[Lemma 3]{zhwb:2022}
$\tr(P_*^{\T}D)<\|P_*^{\T}D\|_{\tr}=\tr\big([P_*Q]^{\T}D\big)$ where $Q$ is an orthogonal polar factor of $P_*^{\T}D$. Hence
$f_{\varepsilon_0}(P_*Q)>f_{\varepsilon_0}(P_*)$, a contradiction, and thus $P_*\in\bbP$.
\end{proof}

We point out that, in the proof above,
updating $P_*$ to $P_*Q$ is rather standard and popularly used lately in \cite{li:2024,luli:2024,wazl:2022a,wazl:2023,zhwb:2022,zhys:2020}.

An self-consistent-field (SCF) iteration
to solve \eqref{eq:opt-OLSR+L21:NEPv-eps} is outlined in \Cref{alg:NEPvSCF4OLSR+L21}.
A few comments regarding its implementation are in order.
\begin{enumerate}[(1)]
  \item A reasonable stopping criterion at Line 2 is
         \begin{equation}\label{eq:stop-1}
         \varepsilon_{\KKT}(P):=\frac {\big\|\scrH_{\varepsilon_0}(P)-P\Lambda_{\varepsilon_0}(P)\big\|_{\F}}
                                   {2 h(P) \big[\|D\|_{\F}+h(P)\,\|A\|_{\F}\big]
                                     +n\alpha}
               \le\epsilon,
         \end{equation}
         where $\Lambda_{\varepsilon_0}(P):=\big(P^{\T}\scrH_{\varepsilon_0}(P)+[\scrH_{\varepsilon_0}(P)]^{\T}P\big)/2$ and $\epsilon$ is a preselected tolerance.
  \item  According to  \Cref{thm:main-NEPv-cvx}, there is no need to compute the partial eigendecomposition at Line~4
         accurately up to the working precision but rather it suffices to make \eqref{eq:NEPv-assume-eps} with a relatively large $\eta>0$.
         This is helpful for overall computational efficiency  when $n$ is large and the partial eigendecomposition is computed iteratively \cite{demm:1997,li:2015,parl:1998,saad:1992}.
\end{enumerate}
Finally with \Cref{thm:main-NEPv-cvx-eps}(a), the general convergence theorems, \cite[Theorems~6.3~and~6.4]{li:2024},
apply. To save space, we omit  stating them here.

\subsection{Acceleration via LOCG}\label{sec:LOCG}
Algorithm~\ref{alg:NEPvSCF4OLSR+L21} involves solving a large eigenvalue problem if there are hugely many features in
sample data.
One option is to employ an iterative eigen-solver. Another option
is to borrow the idea of the locally optimal conjugate gradient technique (LOCG),
which draws inspiration from optimization \cite{poly:1987,taka:1965} and has been
increasingly used for linear systems and eigenvalue problems \cite{beli:2022,imlz:2016,knya:2001,li:2015,yali:2021}
and more recently in \cite{wazl:2022a} for maximizing the sum of coupled traces
and in \cite{li:2024} for general optimization on the Stiefel manifold.

Without loss of generality, let $P^{(-1)}\in\bbO^{n\times k}$ be the approximate maximizer of \eqref{eq:opt-OLSR+L21:work-eps}
from the very previous iterative step, and $P\in\bbO^{n\times k}$ the current approximate maximizer.
We are now looking for the next approximate maximizer
$P^{(1)}$, along the line of LOCG, according to
\begin{equation}\label{eq:LOCG}
P^{(1)}=\arg\max_{Y\in\bbO^{n\times k}}f_{\varepsilon_0}(Y),\,\,\mbox{s.t.}\,\, \cR(Y)\subseteq\cR([P,\scrR(P),P^{(-1)}]),
\end{equation}
where $\cR(\cdot)$ is the column subspace of a matrix, and
$\scrR(P)$ is the gradient of $f_{\varepsilon_0}(\cdot)$ at $P$ with respect to the Stiefel manifold:
\begin{equation}\label{eq:R(P)}
\scrR(P)
=\scrH_{\varepsilon_0}(P)-P\cdot\frac 12\Big[P^{\T}\scrH_{\varepsilon_0}(P)+\scrH_{\varepsilon_0}(P)^{\T}P\Big].
\end{equation}
Initially for the first iteration, we don't have $P^{(-1)}$ and
it is understood that $P^{(-1)}$ is absent from \eqref{eq:LOCG}, i.e.,
simply $\cR(Y)\subseteq\cR([P,\scrR(P)])$.

We still have to numerically solve \eqref{eq:LOCG}. For that purpose, let $W\in\bbO^{n\times m}$ be an orthonormal basis matrix of subspace
$\cR([P,\scrR(P),P^{(-1)}])$. Generically, $m=3k$ but $m<3k$ can happen.
It can be implemented by the Gram-Schmidt orthogonalization process, starting with orthogonalizing the columns of $\scrR(P)$ against $P$ since
$P\in\bbO^{n\times k}$ already. In MATLAB, to fully take advantage of its optimized functions, we simply set
$W=[\scrR(P),P^{(-1)}]$ (or $W=\scrR(P)$ for the first iteration) and then  do
\begin{equation}\label{eq:W-compute}
\framebox{
\begin{minipage}{10cm}
\tt      W=W-P*(P'*W); W=orth(W); W=W-P*(P'*W); W=orth(W);\\
      W=[P,W];
\end{minipage}
}
\end{equation}
where the first line  performs the classical Gram-Schmidt orthogonalization twice to almost ensure that
the resulting  columns of $W$ are fully orthogonal to the columns of $P$ at the end of the first line,
and {\tt orth} is a MATLAB function for orthogonalization, which uses the thin SVD. Another alternative is the thin QR
{\tt [W,$\sim$]=qr(W,0)} and is cheaper than {\tt orth}.
It is important to note that the first $k$ columns of
the final $W$ are the same as those of $P$.

Now it follows from $\cR(Y)\subseteq\cR([P,\scrR(P),P^{(-1)}])=\cR(W)$ that in \eqref{eq:LOCG}
\begin{subequations}\label{eq:LOCGsub}
\begin{equation}\label{eq:LOCGsub:Y}
Y=WZ\quad\mbox{for $Z\in\bbO^{m\times k}$}.
\end{equation}
Problem \eqref{eq:LOCG} becomes
\begin{equation}\label{eq:LOCGsub-1}
Z_{\opt}=\arg\max_{Z\in\bbO^{m\times k}} \wtd f_{\varepsilon_0}(Z),
\end{equation}
where, upon setting $\wtd A=W^{\T}AW$, $\wtd D=W^{\T}D$, and $\bw_i^{\T}=\be_i^{\T}W$,
\begin{equation}\label{eq:LOCGsub-2}
\wtd f_{\varepsilon_0}(Z):=f_{\varepsilon_0}(WZ)
   =\frac {[\tr(Z^{\T}\wtd D)]^2}{\tr(Z^{\T}\wtd AZ)}-\alpha\sum_{i=1}^n\sqrt{\tr(Z^{\T}\bw_i\bw_i^{\T}Z)+\varepsilon_0^2}\,.
\end{equation}
\end{subequations}
Finally $P^{(1)}=WZ_{\opt}$ for \eqref{eq:LOCG}.
Note that $\wtd f_{\varepsilon_0}(Z)$ takes the same form as $f_{\varepsilon_0}(P)$ in \eqref{eq:opt-OLSR+L21:work-eps} and the theory in the previous subsection
for \eqref{eq:opt-OLSR+L21:work-eps} can be extended straightforwardly.
In particular, the KKT condition is
\begin{subequations}\label{eq:KKT-LOCG}
\begin{equation}\label{eq:KKT-LOCG-1}
\wtd\scrH_{\varepsilon_0}(Z)=Z\wtd\Lambda, \quad Z\in\bbO^{n\times k},\quad \wtd\Lambda=\wtd\Lambda^{\T}\in\bbR^{k\times k},
\end{equation}
where
\begin{equation}\label{eq:KKT-LOCG-2}
\wtd\scrH_{\varepsilon_0}(Z):=\frac {\partial \wtd f(Z)}{\partial Z}
    =2 \wtd h(Z) \big[\wtd D-\wtd h(Z)\,\wtd AZ\big]
    -\alpha\sum_{i=1}^n\frac {\bw_i\bw_i^{\T}Z}{\sqrt{\tr(Z^{\T}\bw_i\bw_i^{\T}Z)+\varepsilon_0^2}},
\end{equation}
\end{subequations}
$\wtd h(Z)=\tr(Z^{\T}\wtd D)/\tr(Z^{\T}\wtd AZ)$,
and the associated NEPv is
\begin{subequations}\label{eq:NEPv-LOCG}
\begin{equation}\label{eq:NEPv-LOCG-1}
\wtd H_{\varepsilon_0}(Z)\,Z=Z\wtd\Omega, \quad Z\in\bbO^{m\times k},\quad \wtd\Omega=\wtd\Omega^{\T}\in\bbR^{k\times k},
\end{equation}
where
\begin{equation}\label{eq:NEPv-LOCG-2}
\wtd H_{\varepsilon_0}(Z):=2 \wtd h(Z) \big[\big(\wtd DZ^{\T}+Z\wtd D^{\T})-\wtd h(Z)\,\wtd A\big]
    -\alpha\sum_{i=1}^n\frac {\bw_i\bw_i^{\T}}{\sqrt{\tr(Z^{\T}\bw_i\bw_i^{\T}Z)+\varepsilon_0^2}}
   \in\bbR^{m\times m},
\end{equation}
\end{subequations}
a much smaller matrix.
There are corresponding versions of both \Cref{thm:H(P)-eligibility-eps,{thm:main-NEPv-cvx-eps}} and, in principle,
\Cref{alg:NEPvSCF4OLSR+L21} can be used to solve NEPv~\eqref{eq:NEPv-LOCG}.
\Cref{alg:NEPvLOCG} outlines an accelerating version of \Cref{alg:NEPvSCF4OLSR+L21}.

\begin{algorithm}[t]
\caption{The LOCG-accelerated NEPv approach for solving \eqref{eq:opt-OLSR+L21:work-eps}}
\label{alg:NEPvLOCG}
\begin{algorithmic}[1]
\REQUIRE $0\preceq A\in\bbR^{n\times n}$ (such that $\rank(A)>n-k$), $D\in\bbR^{n\times k}$, regularization parameter $\alpha>0$,
         $\varepsilon_0>0$,
         and initial  $P^{(0)}\in\bbO^{n\times k}$;
\ENSURE  an approximate maximizer of \eqref{eq:opt-OLSR+L21:work-eps}.
\STATE if $\big[P^{(0)}\big]^{\T}D\not\succeq 0$, update $P^{(0)}$ to $P^{(0)}Q_0$ where
       $Q_0=U_0V_0^{\T}$, an orthogonal polar factor of $\big[P^{(0)}\big]^{\T}D=U_0\Sigma_0 V_0$ (SVD);
\STATE $P^{(-1)}=[\,]$; \% null matrix
\FOR{$j=0,1,\ldots$ until convergence}
    \STATE compute $W\in\bbO^{n\times m}$ such that $\cR(W)=\cR(\big[P^{(j)},\scrR(P^{(j)}),P^{(j-1)}\big])$ as in \eqref{eq:W-compute}, where
           $\scrR(P^{(j)})$ is calculated according to \eqref{eq:R(P)};
    \STATE solve \eqref{eq:LOCGsub-1} for $Z_{\opt}$ by \Cref{alg:NEPvSCF4OLSR+L21} with
           inputs $\wtd A$, $\wtd D$, $\wtd H_{\varepsilon_0}(\cdot)$ in \eqref{eq:NEPv-LOCG-2}, initially $Z^{(0)}$
           being the first $k$ columns of $I_m$;
    \STATE $P^{(j+1)}=WZ_{\opt}$;
\ENDFOR
\RETURN the last $P^{(j)}$.
\end{algorithmic}
\end{algorithm}

\begin{remark}\label{rk:SCF4npd+LOCG}
There are a few  comments in order, regarding \Cref{alg:NEPvLOCG}.
\begin{enumerate}[(i)]
  \item The stopping criterion \eqref{eq:stop-1} can be used at Line 3;
  \item It is important to compute $W$ at Line~4 in such a way, as explained moments ago in
        \eqref{eq:W-compute}, that its first $k$
        columns are exactly the same as those of $P^{(j)}$.
        This is because as $P^{(j)}$ converges, $P^{(j+1)}$ changes little from $P^{(j)}$ and hence $Z_{\opt}$
        is increasingly close to the first $k$ columns of $I_m$. This explains the choice of $Z^{(0)}$
        at Line~5.
  \item At Line 5, some saving can be achieved by reusing qualities that are already computed. For example, we may use
        $$
        AP^{(j+1)}=(AW)Z_{\opt}, \quad [P^{(j+1)}]^{\T}AP^{(j+1)}=Z_{\opt}^{\T}(W^{\T}AW)Z_{\opt}
        $$
        to compute
        the next $AP^{(j+1)}$ and $[P^{(j+1)}]^{\T}AP^{(j+1)}$ at  costs $O(nk^2)$ and $O(k^3)$, respectively,
        instead of $O(n^2k)$ by reusing $AW$ and $W^{\T}AW$.
  \item An area of improvement is to solve \eqref{eq:LOCGsub-1} with an accuracy, comparable but fractionally better than the
        current $P^{(j)}$ as an approximate solution of \eqref{eq:opt-OLSR+L21:work-eps}.
        Specifically, if we use
        \eqref{eq:stop-1} at Line~3 here to stop the for-loop: Lines 3--7, with tolerance $\epsilon$, then instead of using the same
        $\epsilon$ for \Cref{alg:NEPvSCF4OLSR+L21} at its Line~2,
        we can use a fraction, say $1/8$,
        of $\varepsilon_{\KKT}(P)$ evaluated at the current approximation $P=P^{(j)}$ as stopping tolerance within \Cref{alg:NEPvSCF4OLSR+L21}.
\end{enumerate}
\end{remark}

%


\begin{remark}\label{rk:PEB-FS}
In subsection~\ref{ssec:mistake-Zhang2018}, we outlined how OLSR-OS21~\eqref{eq:opt-OLSR-OS21} is
alternatingly solved in \cite{Zhang2018} and explained a mistake in its $P$-update calculation, causing
inferior performance by PEB-FS in the next section. We point out that the $P$-update can be
handled analogously by the NEPv approach we laid out in this section upon noticing  that
\eqref{eq:opt-OLSR-OS21} is equivalent to
\begin{equation} \label{eq:opt-OLSR-OS21'}
\max_{P\in\bbO^{n\times k},\,\gamma\in\bbR}
  \Big\{\tilde f(P,\gamma):=-\gamma^2\tr(P^{\T}AP)+2\gamma\tr(P^{\T}D)-\alpha\|P\|_{2,1}\Big\},
\end{equation}
and, with the same $T(\cdot)$ as in \eqref{eq:opt-OLSR+L21:-obj-3},
$
\tilde f(P,\gamma)=\tilde\phi\circ T(P),
$
where
$$
\tilde\phi(\bx;\gamma)=-\gamma^2 x_1+2\gamma x_2 -\alpha\sum_{i=1}^n\sqrt{x_{i+2}}
\quad
\mbox{for $\bx=[x_i]\in\bbR^{n+2}$}.
$$
It can be seen that, given $\gamma$,
$\tilde\phi(\bx;\gamma)$ is convex for $\bx\in\bbR^2\times\bbR_+^n$.
\end{remark}

\section{Numerical Experiments}\label{sec:egs}
As we commented earlier, the role of the $(2,1)$-norm regularization is to induce row sparsity in an optimizer $P$
to the optimization problem \eqref{eq:opt-OCCA+L21} or  \eqref{eq:opt-OLSR-OS21}.
In practice, often, as we do in \cref{sec:OCCA+L21}, $\|P\|_{2,1}$ is perturbed slightly to avoid singularity numerically.
It is done by regularizing the $2$-norms of the rows of $P$
as in subsection~\ref{ssec:practical},
where it is suggested to use $\varepsilon_0=10^{-3}\sqrt{k/n}$.
At the end of computations, an approximate optimizer $P$ is obtained, and any rows of $P$ such that $\|\be_i^{\T}P\|_2\le\varepsilon_0$
is regarded as $0$ numerically for the purpose of feature selection. Practically, the norms of the rows of $P$ are calculated and sorted descendingly, and we then select the $q$ features corresponding to the $q$ largest
row norms of $P$, where $q$ is pre-chosen. Theoretically.
it only makes sense to use $q\ge k$ due to \Cref{thm:not-small}, and yet
practitioners may still select fewer than $k$ features regardless, as in \cite{Zhang2018}.
We shall follow this practice in our experiments, too.


\subsection{Experimental setting}\label{ssec:setting}
Six benchmark datasets are used in our experiments to demonstrate the effectiveness of the proposed feature selection method. They are COIL20\footnote{http://www.cs.columbia.edu/CAVE/software/softlib/coil-20.php},
COIL100\footnote{http://www.cs.columbia.edu/CAVE/software/softlib/coil-100.php},
USPS\footnote{https://www.csie.ntu.edu.tw/$\sim$cjlin/libsvmtools/datasets/multiclass.html\#usps},
Yale\footnote{http://www.cad.zju.edu.cn/home/dengcai/Data/Yale/Yale\_32x32.mat}, YaleB\footnote{http://www.cad.zju.edu.cn/home/dengcai/Data/YaleB/YaleB\_32x32.mat}, and AR\footnote{http://www.cat.uab.cat/Public/Publications/1998/MaB1998/CVCReport24.pdf}, all publicly available online. Their detailed information are summarized in Table~\ref{tab:data}. As our goal  is to rank the importance of input features, we compare our proposed method, which will name OCCA-FS,  with three existing methods:
T-test \cite{Montgomery2007},
CRFS \cite{He2012},
and PEB-FS \cite{Zhang2018},
in terms of classification accuracy.

\begin{table}[h]
	\caption{Statistics of six datasets.} \label{tab:data}
	\centering
	\begin{tabular}{l|cccccc}
		\hline
		 & COIL20 & COIL100 & USPS & Yale & AR & YaleB\\\hline
		number of data samples ($p$) & 1440 & 7200 & 9298 & 165 & 840 & 2414\\
		number of features ($n$) & 1024 & 1024 & 256 & 1024 & 768 & 1024\\
		number of class labels ($k$) & 20 & 100 & 10 & 15  & 120 & 38\\\hline
	\end{tabular}
\end{table}

Our experiments are carried as follows. For each  dataset, we first split it randomly into two subsets of 60\%
for training and 40\% for testing, and then run each compared method on the training subset to generate
a ranking of the input features, finally train a one-nearest neighbor classifier on the training subset
with the $q$ top-ranked features and calculate the classification accuracy on the testing subset with the $q$ top-ranked features. We let $q$ vary in $[10, 20, 30 ,40, 50]$ as in \cite{Zhang2018}. The same experiment is repeated ten times with different random splitting  of 60\% for training and 40\% for testing.
We report the average accuracy with standard deviation by each method in \Cref{tab:accuracy}.

In what follows, we will compare the effectiveness of feature ranking by the four methods through
classification accuracy in subsection~\ref{sec:classifcation}, and then evaluate the optimization solver for solving \eqref{eq:opt-OLSR-OS21}
in PEB-FS \cite{Zhang2018} and our NEPv solver for \eqref{eq:opt-OCCA+L21} in OCCA-FS, respectively,
in terms of empirical convergence, in subsection~\ref{sec:convergence}, and finally in subsection~\ref{ssec:para-sensitivity},
we will investigate the parameter sensitivity analysis of our OCCA-FS
with respect to hyper-parameter $\alpha$.

\subsection{Classification performance} \label{sec:classifcation}
In this subsection, we evaluate four feature selection methods:
T-test \cite{Montgomery2007}, CRFS \cite{He2012}, PEB-FS \cite{Zhang2018} and our newly developed OCCA-FS,
in terms of classification accuracy. For each method on a dataset,
upon each random splitting into two subsets: 60\% for training and 40\% for testing, it first uses
the training subset to generate a feature ranking for the dataset, and then a classifier on the $q$ top-ranked features
is trained, and finally, the classifier is tested on the testing subset. We report in Table~\ref{tab:accuracy} the average accuracy with standard deviation by all methods.
%
We have the following observations:
\begin{enumerate}[(i)]
  \item  OCCA-FS significantly outperforms the others on three datasets COIL1000, USPS and AR. OCCA-FS is competitive to T-test and CRFS on COIL20 and Yale.  When a small number of the top-ranked features are used, OCCA-FS shows better performance than others. This is important because, conceivably,
      for the same performance, the smaller the number of selected features the more preferable.
  \item CRFS is competitive to OCCA-FS on two of the six datasets.  CRFS works better than others on YaleB and marginally better than OCCA-FS on COIL20.
  \item OCCA-FS significantly outperforms PEB-FS on all six datasets for five different $q$, the number of selected features, except for USPS for the case of 10 selected features, despite that both methods are essentially based
      on the same mathematical principle because \eqref{eq:opt-OCCA+L21}
can be viewed as the one obtained from \eqref{eq:opt-OLSR-OS21} by optimizing scaling parameter $\gamma$ out.
\end{enumerate}
To the last observation, we believe, it is  due to a  mistake in the optimization procedure in PEB-FS \cite{Zhang2018},
as we previously explained in subsection~\ref{ssec:mistake-Zhang2018}; see also subsection~\ref{ssec:para-sensitivity} below.

\begin{table}[t]
	\caption{\footnotesize Average accuracy  with standard deviation
          over 10 random splitting of 60\% for training and 40\% for testing as the number $q$ of selected features varies
          from 10 to 50. The best accuracies are  in {\bf bold}.} \label{tab:accuracy}
\begin{scriptsize}
	\begin{tabular}{clcccc}
		\hline
		 $q$ features& dataset& T-test & CRFS & PEB-FS & OCCA-FS \\\hline
		\multirow{ 6}{*}{10} & COIL20 & 0.7189 $\pm$ 0.0424 & 0.7128 $\pm$ 0.0484 & 0.3639 $\pm$ 0.0208 & \textbf{0.8521 $\pm$ 0.0310}\\
		& COIL100 & 0.1752 $\pm$ 0.0438 & 0.3450 $\pm$ 0.0674 & 0.2255 $\pm$ 0.0221 & \textbf{0.4824 $\pm$ 0.1092}\\
		& USPS & 0.6831 $\pm$ 0.0365 & 0.7303 $\pm$ 0.0198 & \textbf{0.7512 $\pm$ 0.0344} & 0.6998 $\pm$ 0.0587\\
		& Yale & \textbf{0.4379 $\pm$ 0.0924} & 0.3879 $\pm$ 0.0535 & 0.3197 $\pm$ 0.0471 & 0.3970 $\pm$ 0.0748\\
		& AR & 0.2866 $\pm$ 0.0688 & 0.1899 $\pm$ 0.0330 & 0.2607 $\pm$ 0.0379 & \textbf{0.3482 $\pm$ 0.0347}\\
		& YaleB & 0.4366 $\pm$ 0.0264 & \textbf{0.5041 $\pm$ 0.0374} & 0.3879 $\pm$ 0.0611 & 0.4715 $\pm$ 0.0431\\
		\hline \hline
		\multirow{ 6}{*}{20} & COIL20 & 0.8443 $\pm$ 0.0480 & 0.9196 $\pm$ 0.0296 & 0.4188 $\pm$ 0.0308 & \textbf{0.9453 $\pm$ 0.0279}\\
		& COIL100 & 0.3387 $\pm$ 0.0504 & 0.5662 $\pm$ 0.0711 & 0.3145 $\pm$ 0.0185 & \textbf{0.6460 $\pm$ 0.1038}\\
		& USPS & 0.7730 $\pm$ 0.0058 & 0.8732 $\pm$ 0.0197 & 0.8765 $\pm$ 0.0152 & \textbf{0.8828 $\pm$ 0.0218}\\
		& Yale & \textbf{0.4909 $\pm$ 0.0806} & 0.4167 $\pm$ 0.0501 & 0.3894 $\pm$ 0.0452 & 0.4409 $\pm$ 0.0578\\
		& AR & 0.3563 $\pm$ 0.0588 & 0.2979 $\pm$ 0.0334 & 0.3086 $\pm$ 0.0217 & \textbf{0.4268 $\pm$ 0.0259}\\
		& YaleB & 0.5576 $\pm$ 0.0241 & \textbf{0.6345 $\pm$ 0.0251} & 0.5489 $\pm$ 0.0463 & 0.5503 $\pm$ 0.0439\\
		\hline \hline
		\multirow{ 6}{*}{30} & COIL20 & 0.9012 $\pm$ 0.0202 & \textbf{0.9648 $\pm$ 0.0116} & 0.9521 $\pm$ 0.0418 & 0.9630 $\pm$ 0.0130\\
		& COIL100 & 0.4517 $\pm$ 0.0351 & 0.6885 $\pm$ 0.0270 & 0.3636 $\pm$ 0.0206 & \textbf{0.7768 $\pm$ 0.0424}\\
		& USPS & 0.8532 $\pm$ 0.0031 & 0.9205 $\pm$ 0.0070 & 0.9216 $\pm$ 0.0075 & \textbf{0.9343 $\pm$ 0.0123}\\
		& Yale & \textbf{0.5000 $\pm$ 0.1010} & 0.4288 $\pm$ 0.0411 & 0.4273 $\pm$ 0.0301 & 0.4803 $\pm$ 0.0441\\
		& AR & 0.3836 $\pm$ 0.0416 & 0.3292 $\pm$ 0.0247 & 0.3461 $\pm$ 0.0278 & \textbf{0.4473 $\pm$ 0.0378}\\
		& YaleB & 0.6083 $\pm$ 0.0221 & \textbf{0.7086 $\pm$ 0.0194} & 0.5982 $\pm$ 0.0258 & 0.5968 $\pm$ 0.0355\\
		\hline \hline
		\multirow{ 6}{*}{40} & COIL20 & 0.9266 $\pm$ 0.0160 & \textbf{0.9753 $\pm$ 0.0100} & 0.9741 $\pm$ 0.0137 & 0.9734 $\pm$ 0.0064\\
		& COIL100 & 0.5354 $\pm$ 0.0288 & 0.7326 $\pm$ 0.0158 & 0.3960 $\pm$ 0.0109 & \textbf{0.8406 $\pm$ 0.0191}\\
		& USPS & 0.8926 $\pm$ 0.0055 & 0.9387 $\pm$ 0.0039 & 0.9398 $\pm$ 0.0025 & \textbf{0.9547 $\pm$ 0.0049}\\
		& Yale & \textbf{0.5136 $\pm$ 0.0902} & 0.4576 $\pm$ 0.0478 & 0.4515 $\pm$ 0.0473 & 0.5015 $\pm$ 0.0454\\
		& AR & 0.4161 $\pm$ 0.0576 & 0.3735 $\pm$ 0.0253 & 0.3836 $\pm$ 0.0185 & \textbf{0.4524 $\pm$ 0.0450}\\
		& YaleB & 0.6424 $\pm$ 0.0216 & \textbf{0.7581 $\pm$ 0.0205} & 0.5262 $\pm$ 0.0282 & 0.6445 $\pm$ 0.0287\\
		\hline \hline
		\multirow{ 6}{*}{50} & COIL20 & 0.9472 $\pm$ 0.0126 & \textbf{0.9807 $\pm$ 0.0090} & 0.8694 $\pm$ 0.0125 & 0.9793 $\pm$ 0.0097\\
		& COIL100 & 0.5849 $\pm$ 0.0252 & 0.7570 $\pm$ 0.0233 & 0.4053 $\pm$ 0.0069 & \textbf{0.8835 $\pm$ 0.0219}\\
		& USPS & 0.9176 $\pm$ 0.0047 & 0.9489 $\pm$ 0.0026 & 0.9476 $\pm$ 0.0044 & \textbf{0.9609 $\pm$ 0.0032}\\
		& Yale & \textbf{0.5379 $\pm$ 0.0719} & 0.4455 $\pm$ 0.0590 & 0.4545 $\pm$ 0.0258 & 0.4955 $\pm$ 0.0378\\
		& AR & 0.4354 $\pm$ 0.0641 & 0.3884 $\pm$ 0.0329 & 0.4039 $\pm$ 0.0125 & \textbf{0.4628 $\pm$ 0.0469}\\
		& YaleB & 0.6500 $\pm$ 0.0156 & \textbf{0.7780 $\pm$ 0.0167} & 0.5428 $\pm$ 0.0216 & 0.6697 $\pm$ 0.0234\\
		\hline
	\end{tabular}
\end{scriptsize}
\end{table}

\subsection{Empirical illustration on convergence by PEB-FS and OCCA-FS} \label{sec:convergence}
\Cref{tab:accuracy} clearly shows that
our  OCCA-FS outperforms PEB-FS on all six datasets, and yet the two methods are essentially based on
the same mathematical principle but solved differently. Previously, we pointed out a mistake
in the optimization procedure in \cite{Zhang2018} and argued that should have a lot to do with
the inferior performance by PEB-FS. In this subsection, we will numerically demonstrate
how objective value moves during the iterative processes by the our NEPv approach  in \cref{sec:OCCA+L21}
and by the alternating solver in PEB-FS \cite{Zhang2018}.
%
Specifically, Figure~\ref{fig:convergence} plots one of the ten repeated runs for $\alpha=0.01$.
We observe that
\begin{enumerate}[(i)]
  \item OCCA-FS always ends up with a better optimizer than PEB-FS in terms of objective value  for all six datasets;
  \item OCCA-FS enjoys monotonic convergence in objective value as guaranteed by
        \cite[Theorems~6.3 and 6.4]{li:2024}; however,
        PEB-FS behaves erroneously for dataset Yale, contradicting \cite[Theorem 3.2]{Zhang2018} which claims
        monotonicity in objective value by PEB-FS.
\end{enumerate}
These observations show that our OCCA-FS embeds a robust optimization solver
that is globally convergent, while PEB-FS \cite{Zhang2018} does not.

\begin{figure}[!ht]
	\centering
	\begin{tabular}{ccc}
		\hline
		{\footnotesize COIL20} & {\footnotesize COIL100} & {\footnotesize USPS} \\
		\includegraphics[width=0.28 \textwidth]{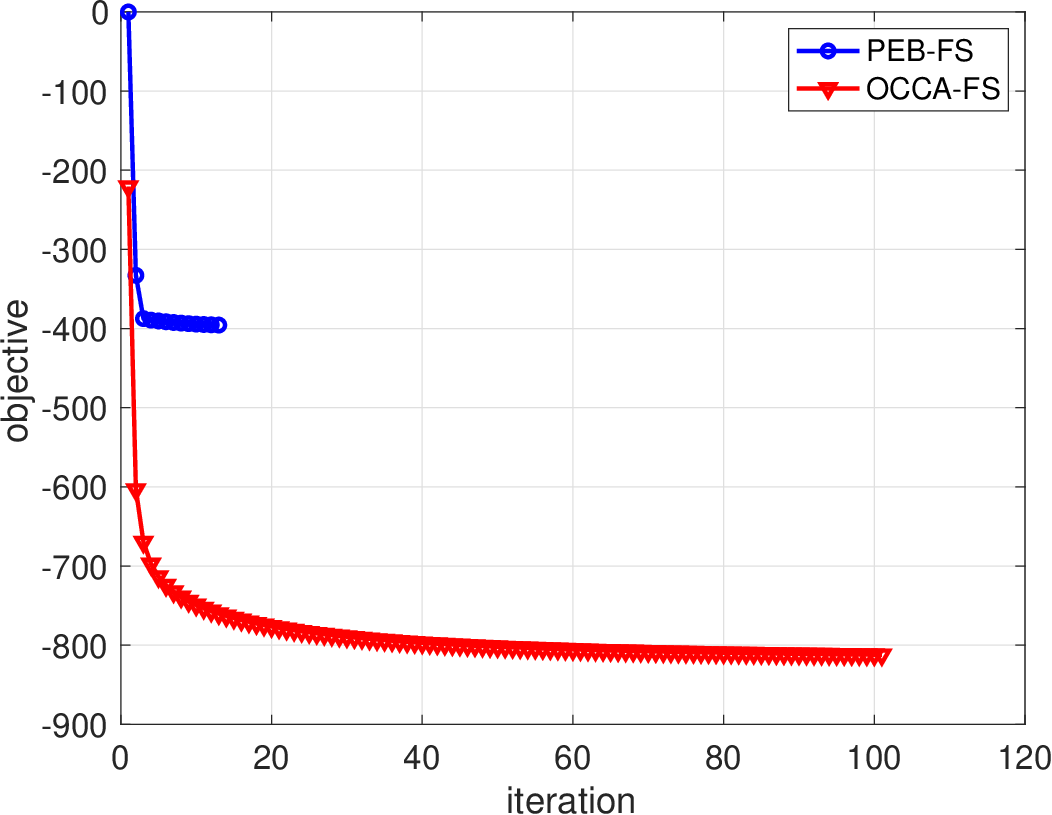}
          &	\includegraphics[width=0.28 \textwidth]{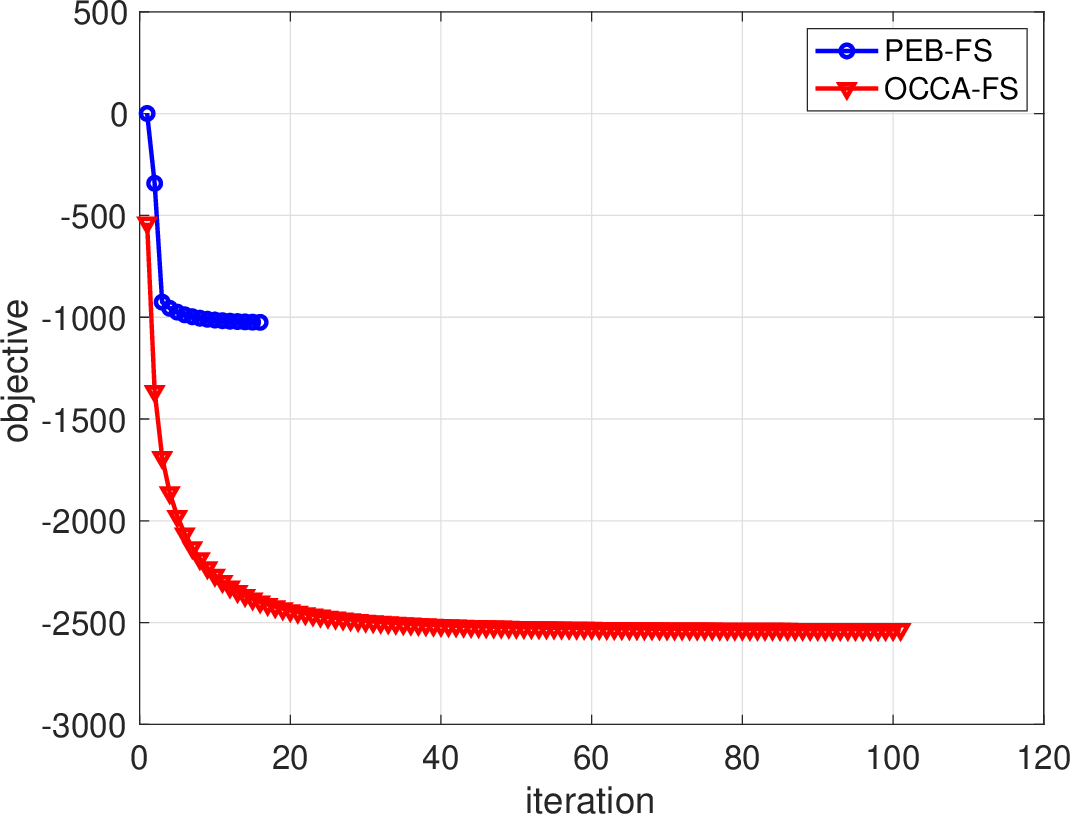}
          & \includegraphics[width=0.28 \textwidth]{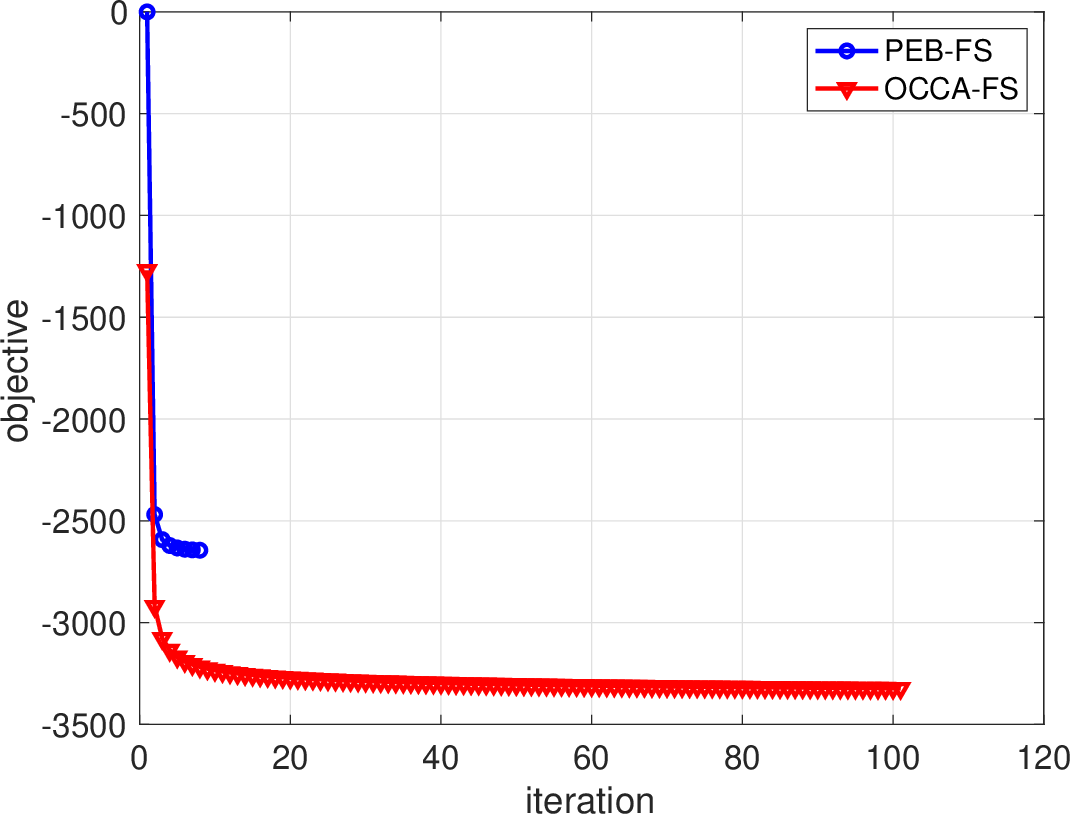} \\ \hline
		 {\footnotesize Yale} & {\footnotesize AR} & {\footnotesize YaleB} \\
		\includegraphics[width=0.28 \textwidth]{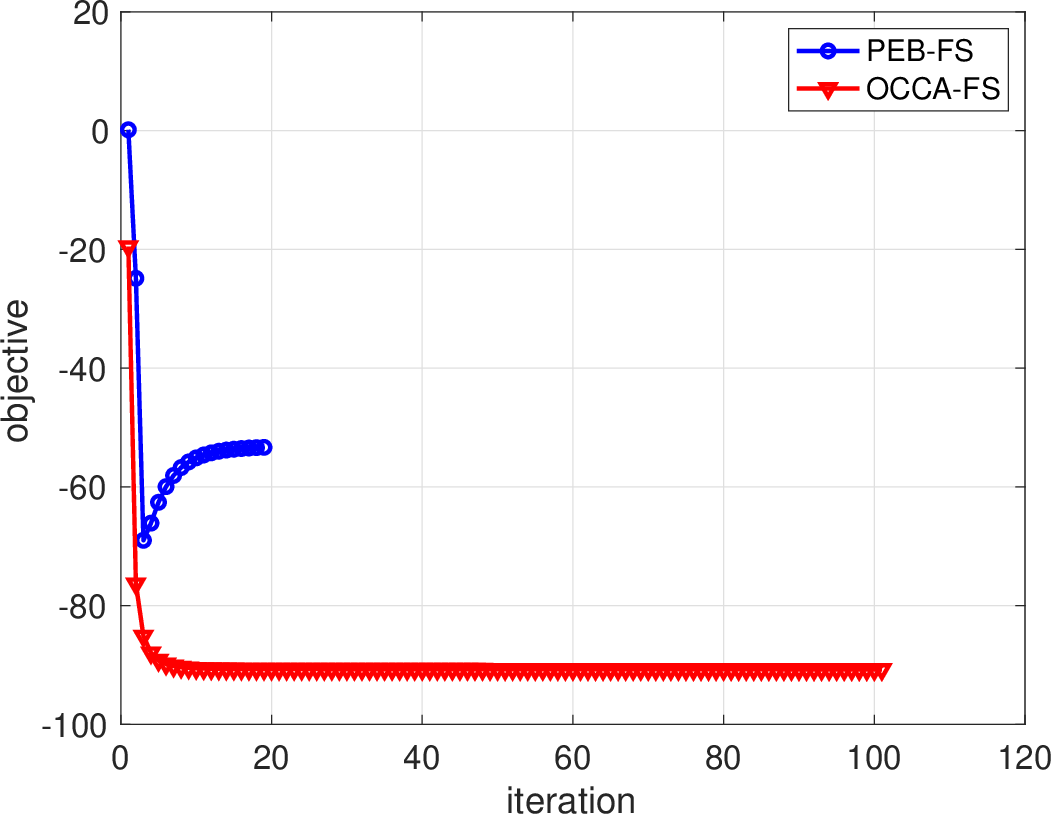}
          & \includegraphics[width=0.28 \textwidth]{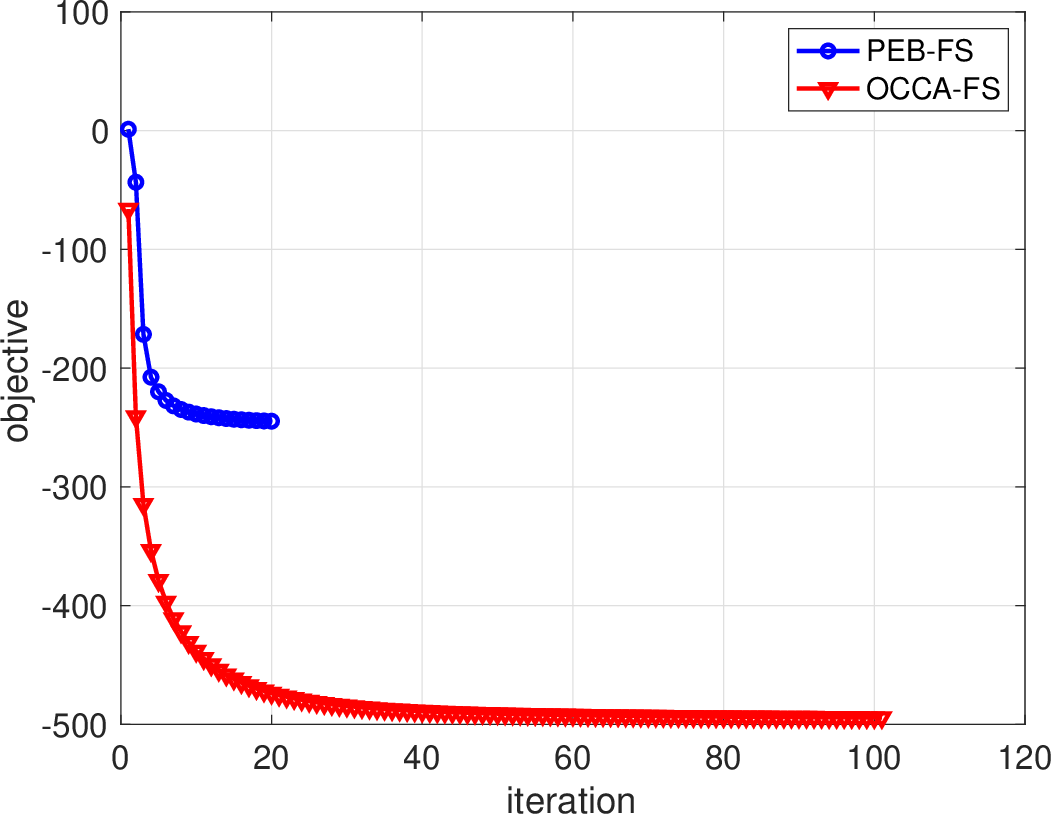}
          & \includegraphics[width=0.28 \textwidth]{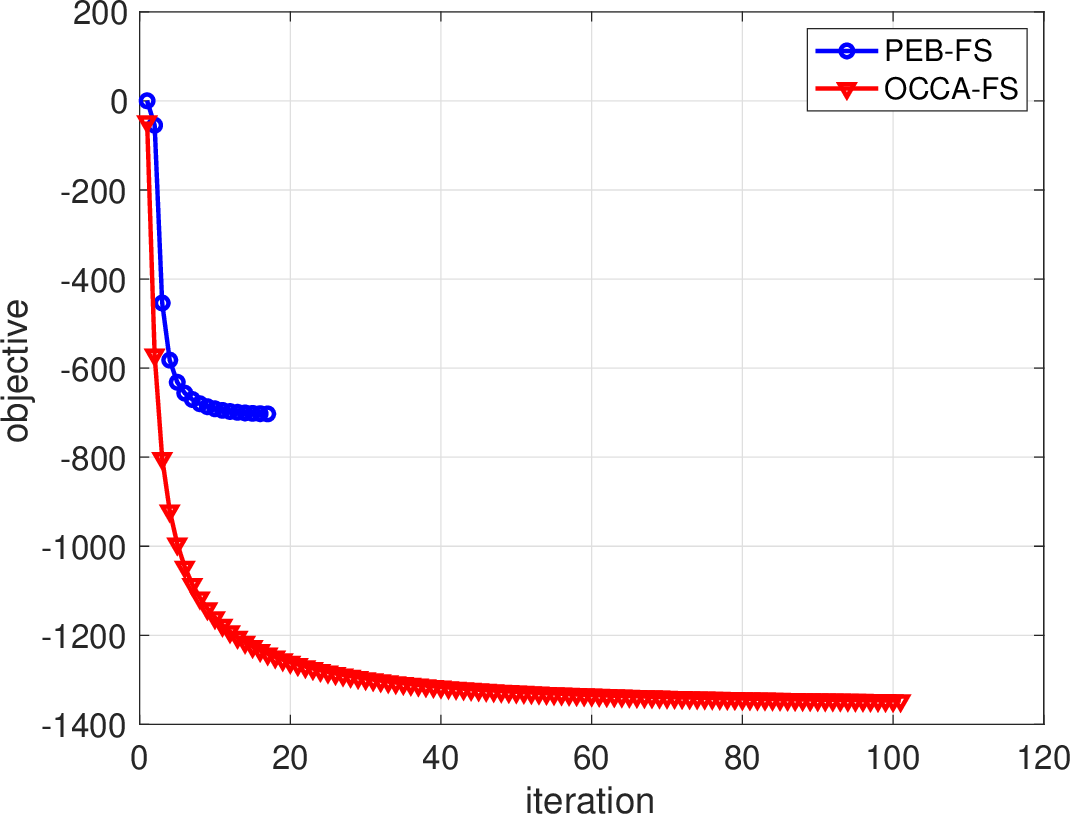} \\
          \hline
	\end{tabular}
	\caption{\footnotesize Behavior of objective values during optimization processes in OCCA-FS and PEB-FS.
     In terms of optimization, eventually, the smaller the objective value is, the more superior the optimization solver will be.
     For dataset Yale,
     PEB-FS does not produce monotonically decreasing objective values as claimed by \cite[Theorem 3.2]{Zhang2018}.} \label{fig:convergence}
\end{figure}


\subsection{Parameter sensitivity analysis}\label{ssec:para-sensitivity}
Both mathematical models \eqref{eq:opt-OCCA+L21} and \eqref{eq:opt-OLSR-OS21} contain a hyperparmeter $\alpha$, which is used to control the magnitudes of the 2-norm of the rows in optimizer $P$ so as to rank input features. In order to analyze its sensitivity, we
perform a sensitivity analysis based on evaluating the resulting classifiers on varying the number of selected features.
Hence, besides $\alpha$, the number of selected features is considered as another hyperparemter for this analysis. Specifically, we run our OCCA-FS with $\alpha \in [0.01, 0.05, 0.1, 1, 10, 100]$ and the number of selected features in $[10, 20, 30, 40, 50]$.
The average accuracy by our OCCA-FS over $10$ repeated experiments with respect to $\alpha$ and the number of selected features is shown in Figure~\ref{fig:parameter}. For a fixed number of selected features, the smaller $\alpha$ is, the worse the testing accuracy is on COIL20, YaleB, and COIL100, but the trend is different on USPS, AR, and Yale.
For a small number of selected features, the accuracy is sensitive to $\alpha$, depending on the datasets. However, when the number of selected features become large, the accuracy is less sensitive to $\alpha$.

\begin{figure}[!ht]
	\centering
	\begin{tabular}{ccc}
		\hline
		{\footnotesize COIL20} & {\footnotesize COIL100} & {\footnotesize USPS}\\
		\includegraphics[width=0.28 \textwidth]{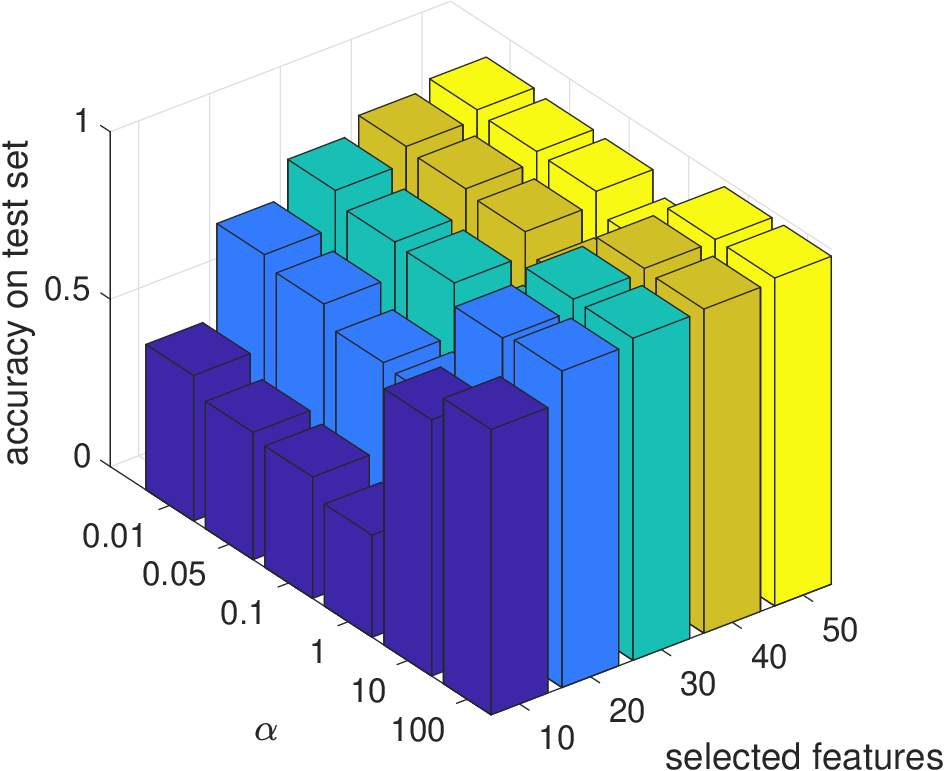}
          & \includegraphics[width=0.28 \textwidth]{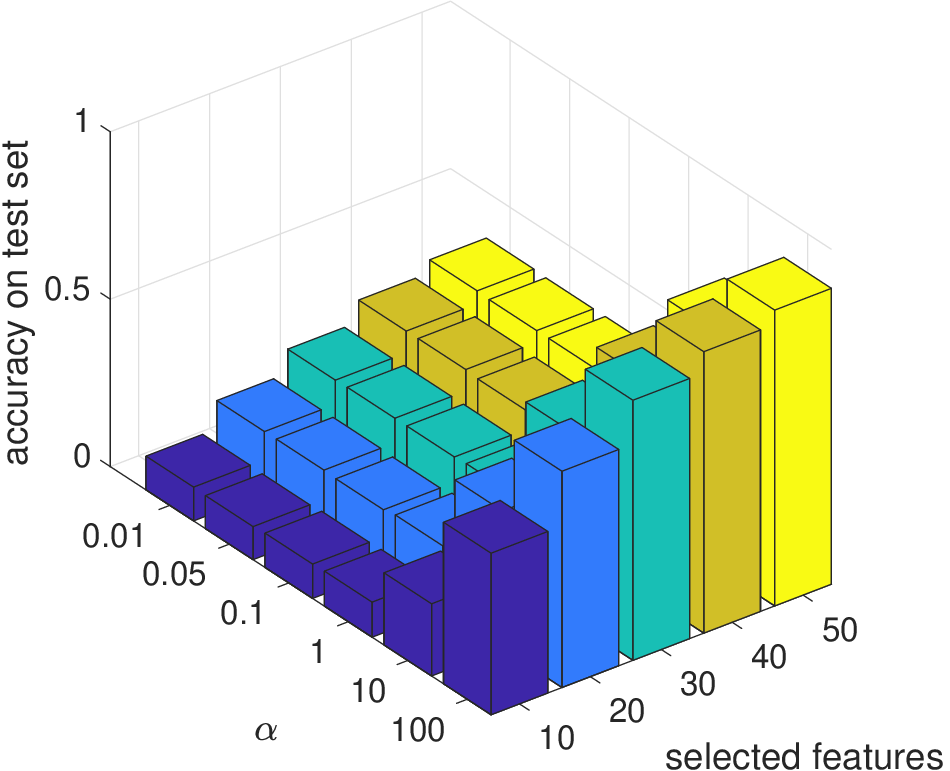}
          & \includegraphics[width=0.28 \textwidth]{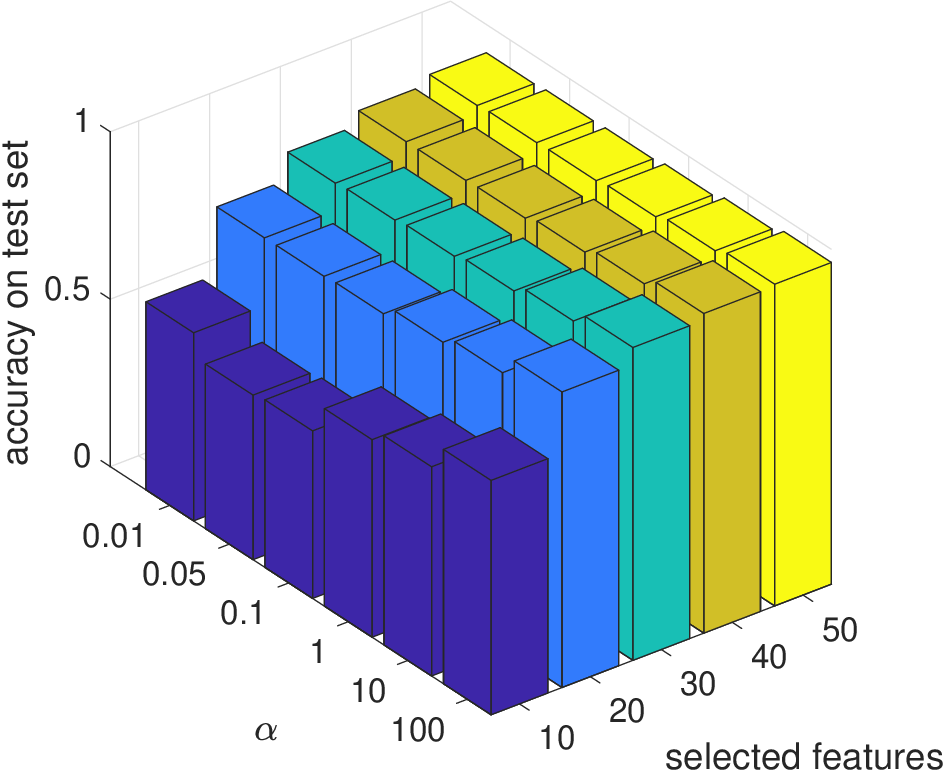} \\ \hline
		{\footnotesize Yale} & {\footnotesize 	AR} & {\footnotesize YaleB} \\
		\includegraphics[width=0.28 \textwidth]{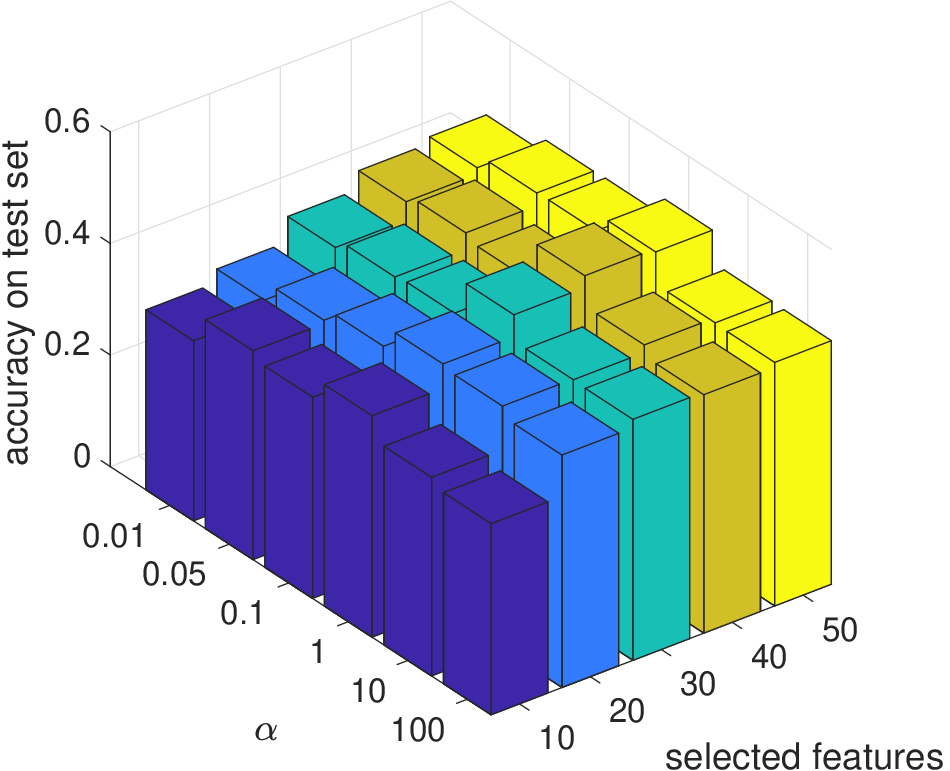}
		  & \includegraphics[width=0.28 \textwidth]{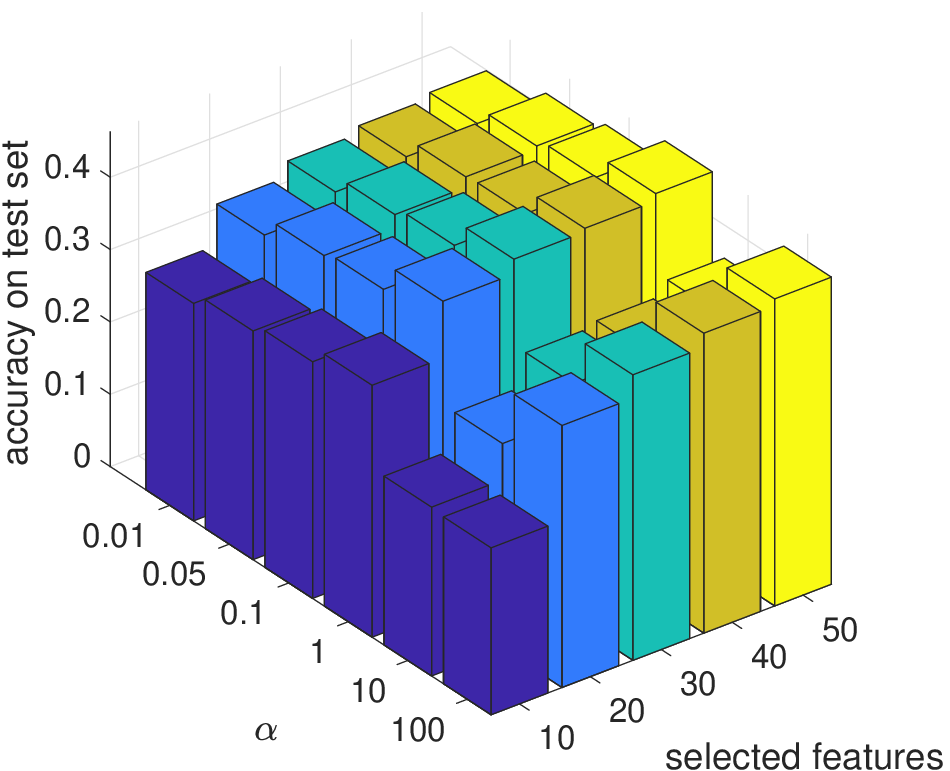}
          & \includegraphics[width=0.28 \textwidth]{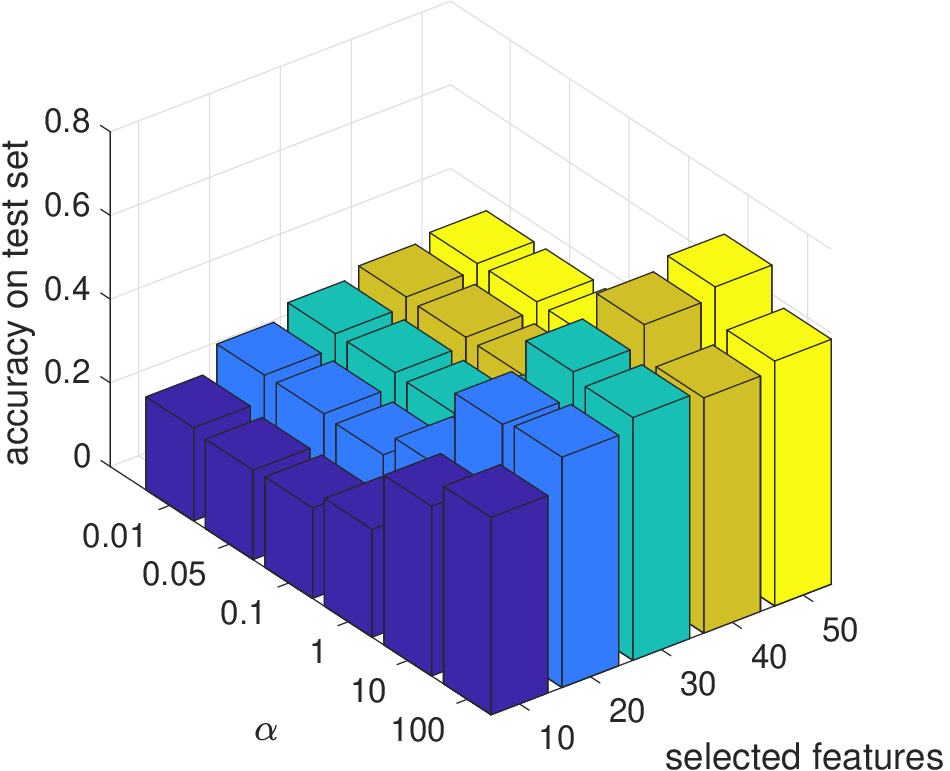} \\ \hline
	\end{tabular}
	\caption{\footnotesize Parameter sensitivity analysis on OCCA-FS with respect to testing accuracy via varying the number $q$ of selected features and regularization parameter $\alpha$ on the six datasets.} \label{fig:parameter}
\end{figure}


\subsection{Comparisons of plain NEPv with accelerated NEPv}
As the computation complexity
for solving \eqref{eq:opt-OCCA+L21}
is affected by the number of input features, we now evaluate our plain NEPv (or NEPv for short)
in \Cref{alg:NEPvSCF4OLSR+L21}
against the LOCG-accelerated NEPv (or AccNEPv for short) in \Cref{alg:NEPvLOCG} in terms of computation time and accuracy with respect to varying the number of input features.
For the task of feature selection, given a dataset, we simulate a collection of similar datasets by adding random noisy features and thereby without introducing new meaningful features.
In this experiment, we take COIL20 as the base dataset, and create a pool of datasets named COIL20-$t$ by appending $t \times 1000$ noisy features randomly drawn from the uniform distribution in the interval $(0,0.01)$, and so the total number of input features for COIL20-$t$ is $1024 + 1000 t$. The evaluation criteria on these datasets are the same
as in subsections~\ref{sec:classifcation} and \ref{sec:convergence}, while the main focus of  comparison
is computational time by  NEPv and AccNEPv in solving \eqref{eq:opt-OCCA+L21}.

\begin{figure}[!ht]
	\centering
\begin{tabular}{ccc}
\hline
{\footnotesize COIL20-1} & {\footnotesize COIL20-2} & {\footnotesize COIL20-3}\\
\includegraphics[width=0.28 \textwidth]{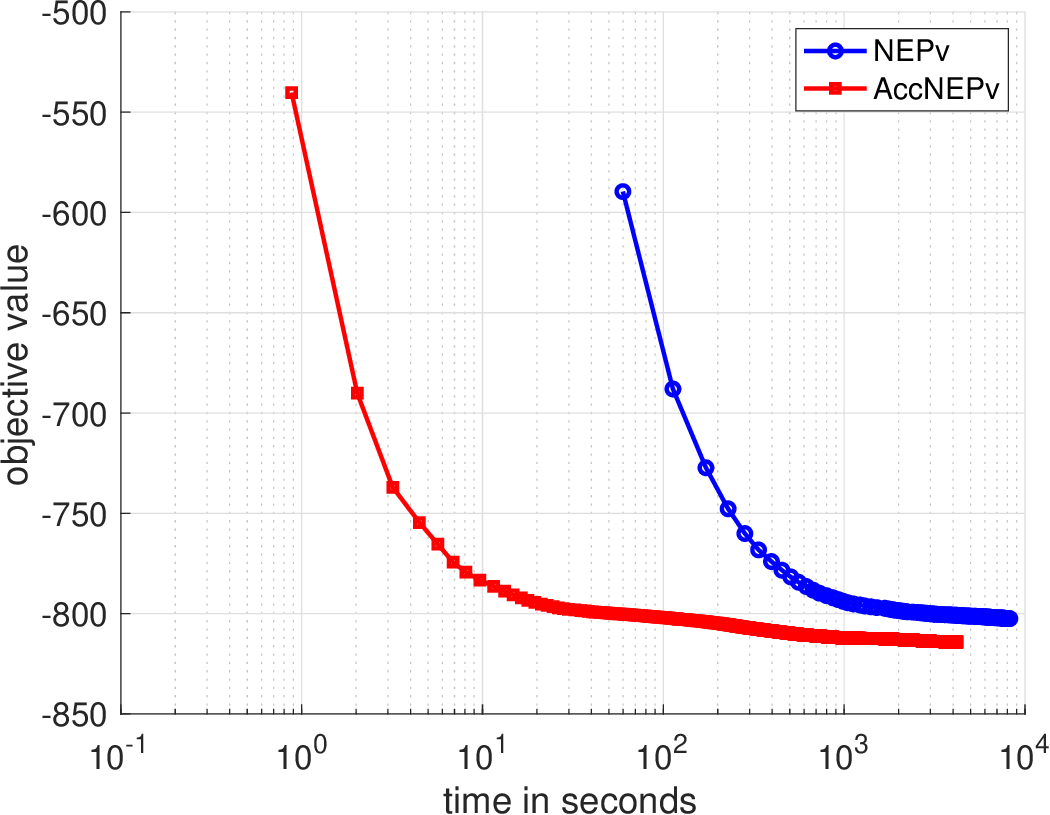} & \includegraphics[width=0.28 \textwidth]{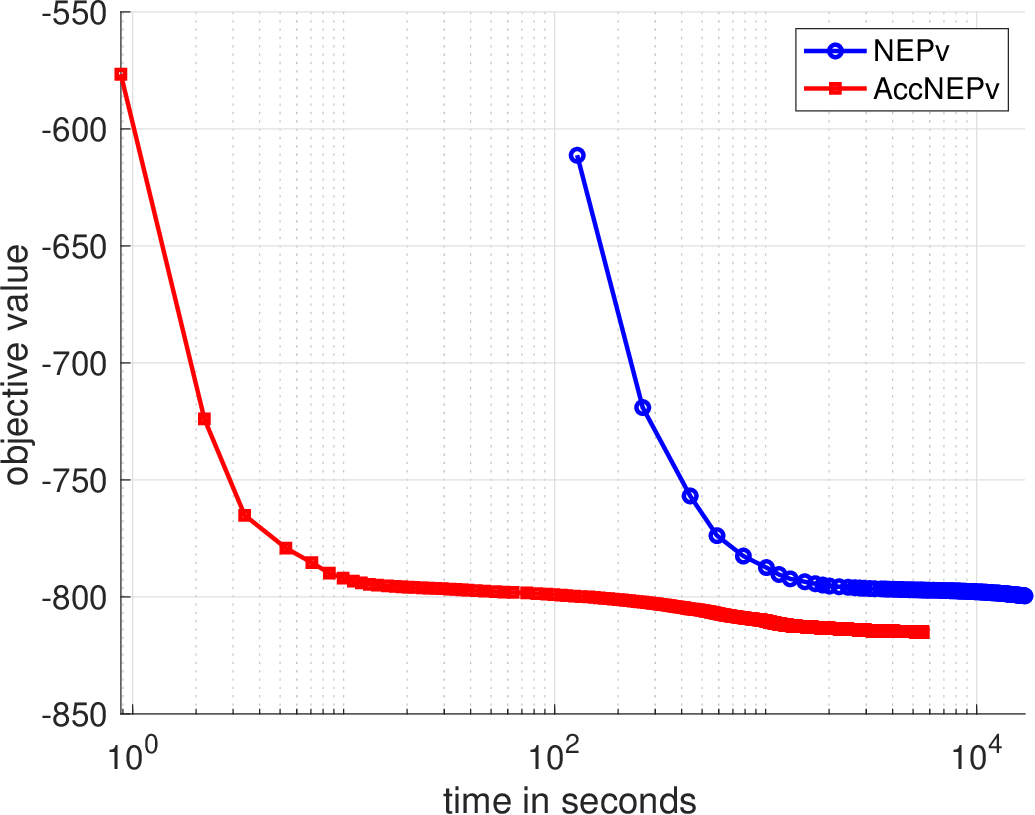} &
\includegraphics[width=0.28 \textwidth]{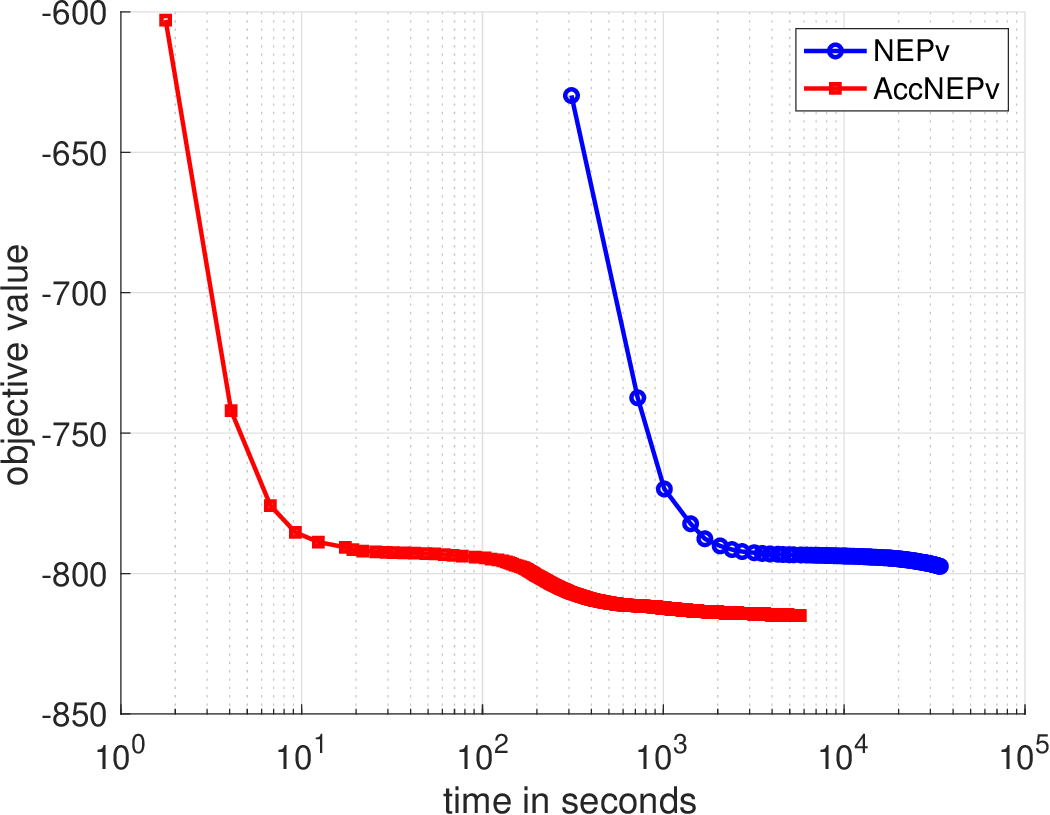} \\
\includegraphics[width=0.28 \textwidth]{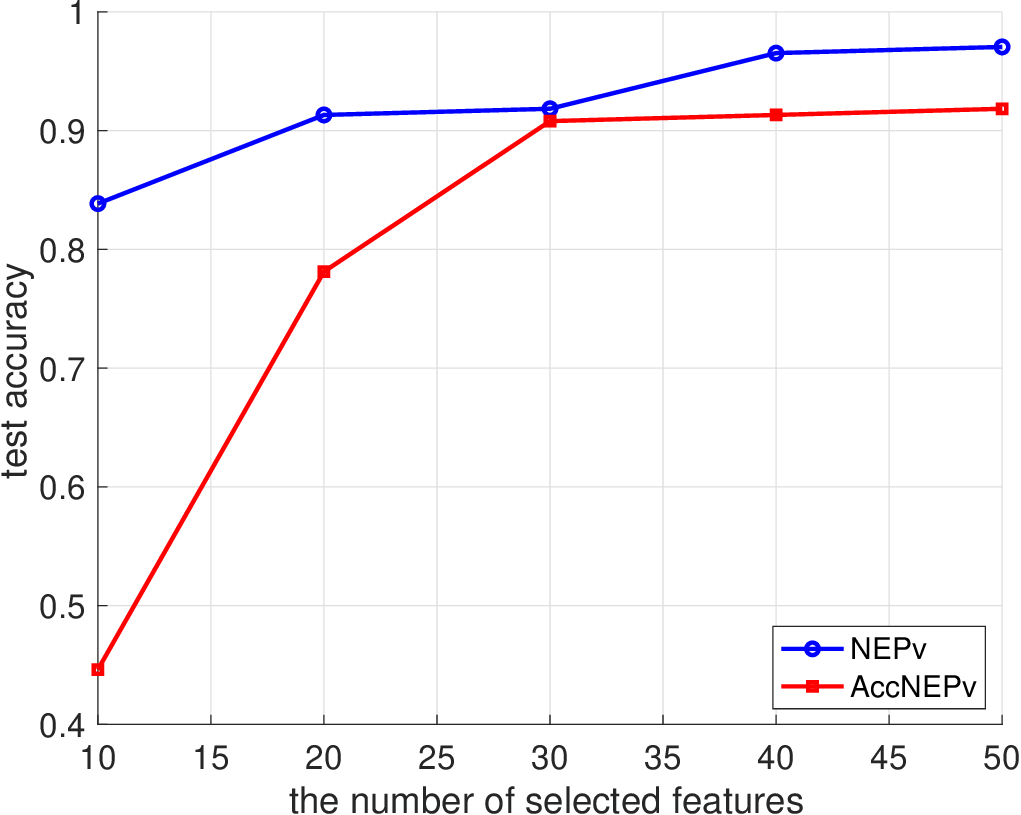} &
\includegraphics[width=0.28 \textwidth]{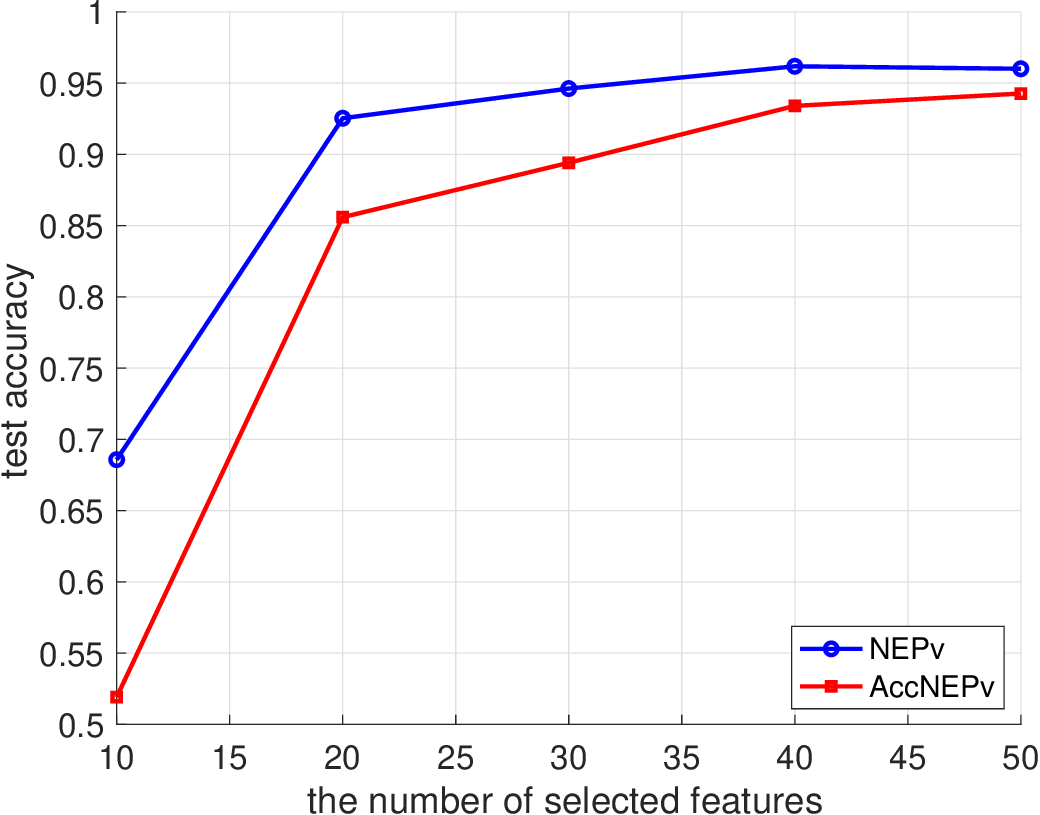} &
\includegraphics[width=0.28 \textwidth]{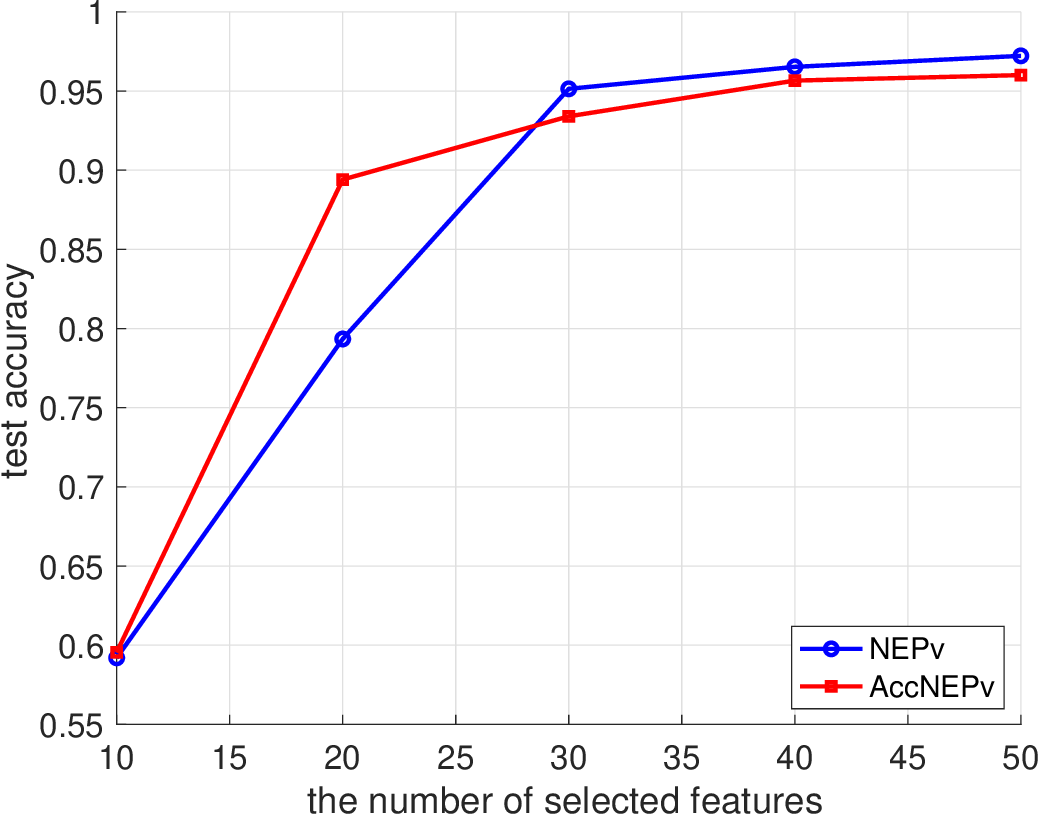} \\
\hline
{\footnotesize COIL20-4} & {\footnotesize COIL20-5} & {\footnotesize COIL20-6}\\
\includegraphics[width=0.28 \textwidth]{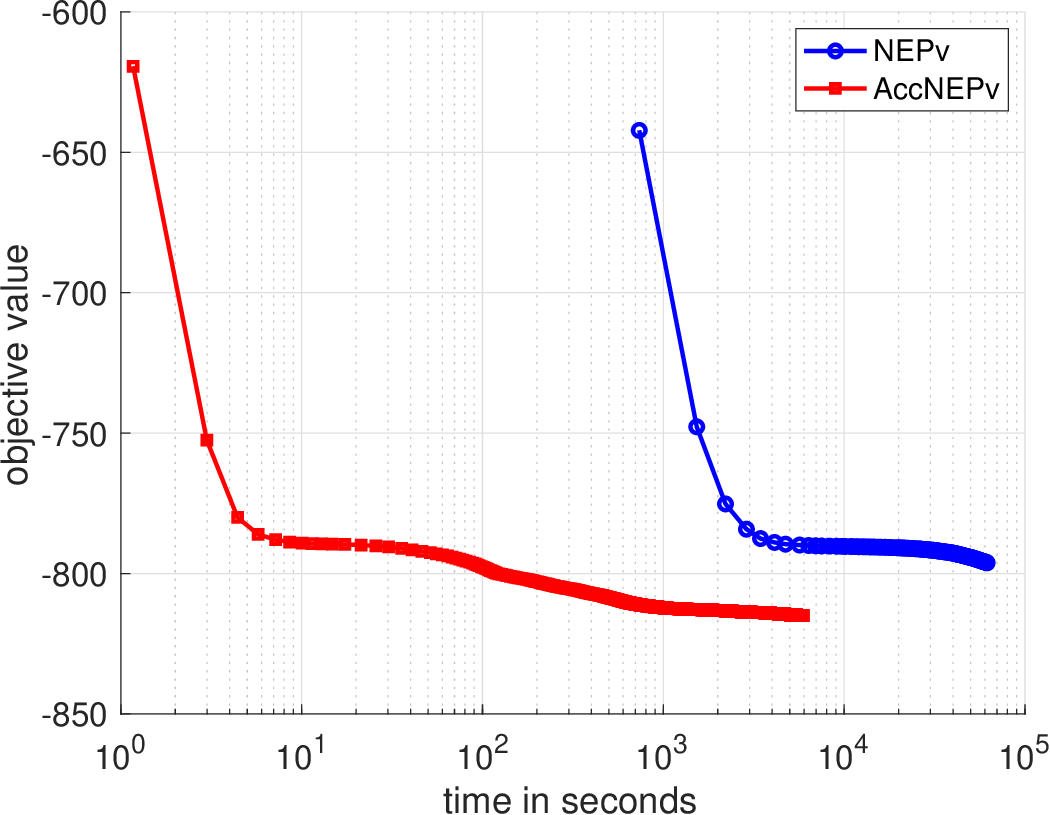} & \includegraphics[width=0.28 \textwidth]{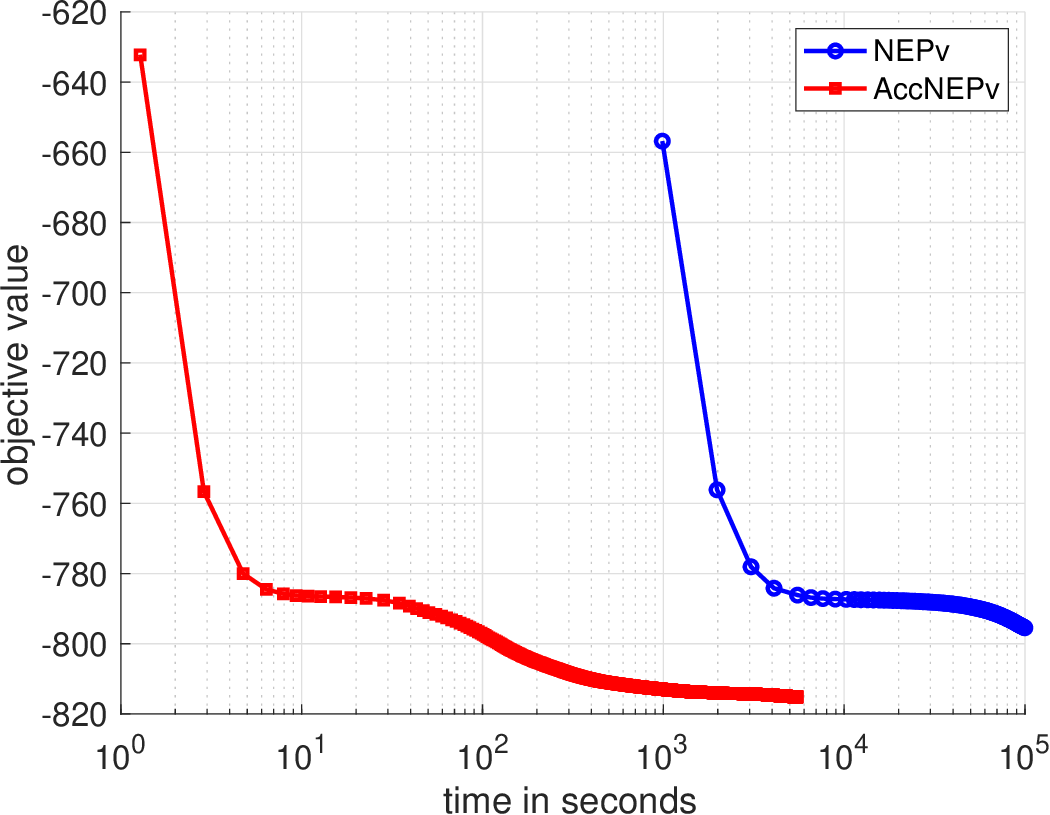} &
\includegraphics[width=0.28 \textwidth]{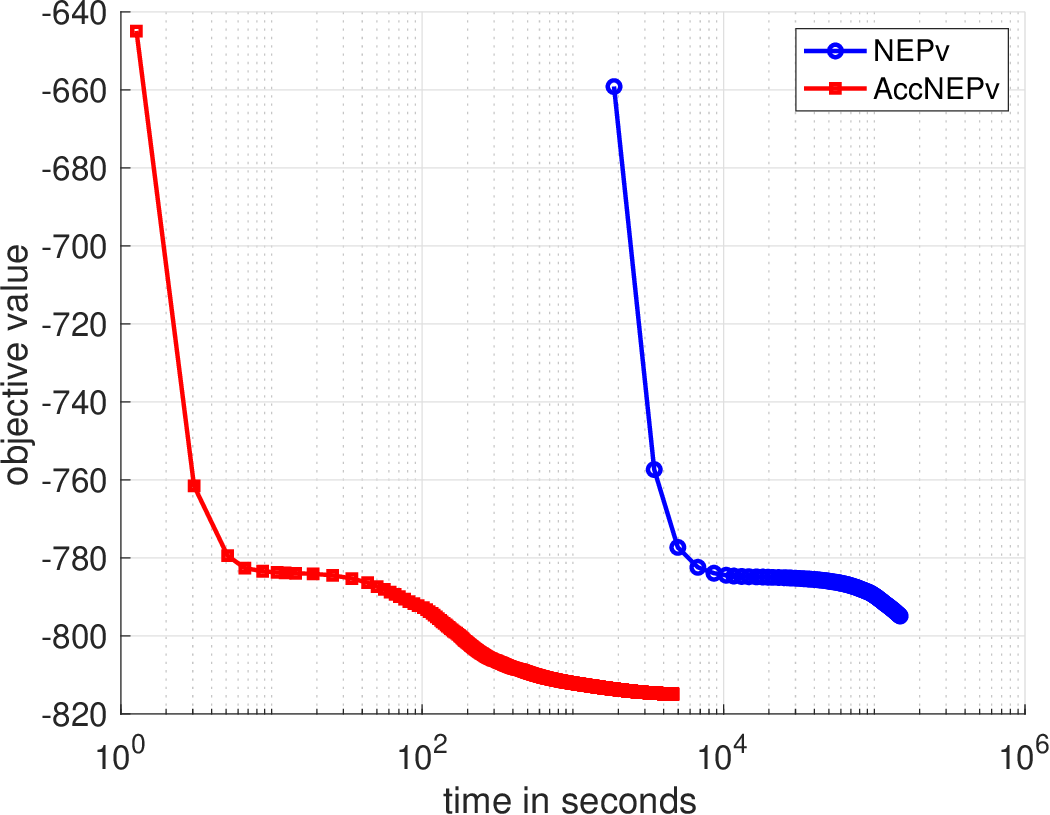} \\
\includegraphics[width=0.28 \textwidth]{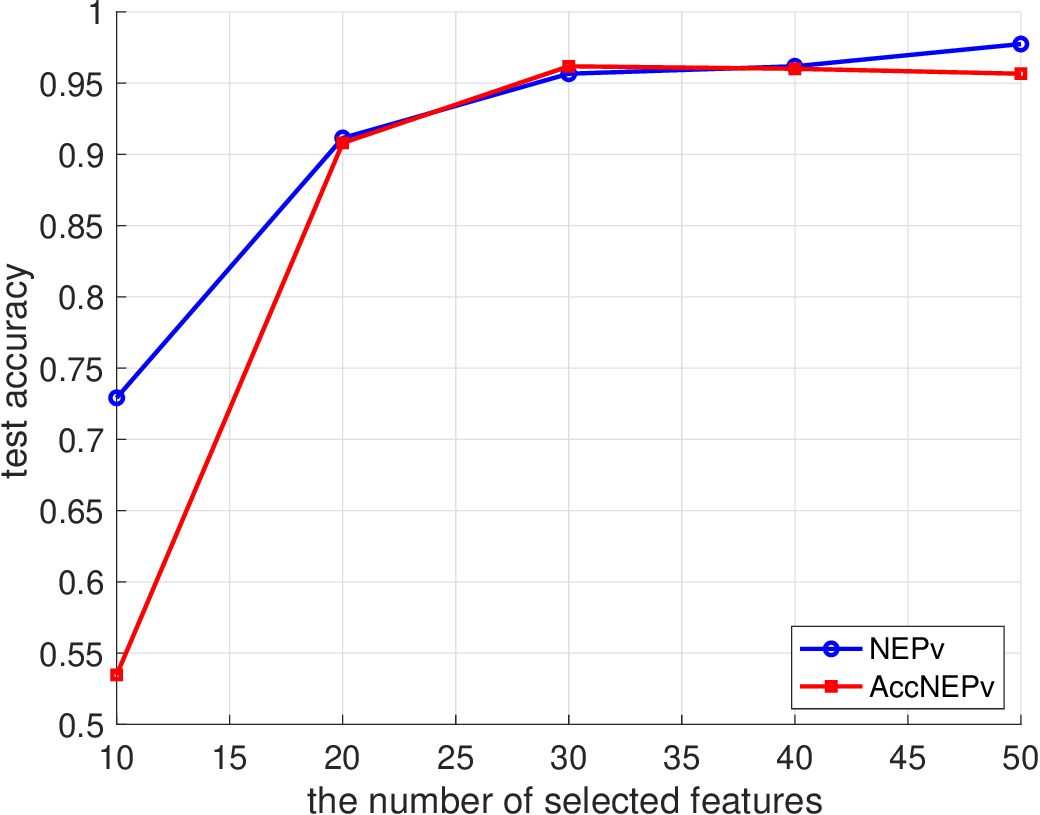} &
\includegraphics[width=0.28 \textwidth]{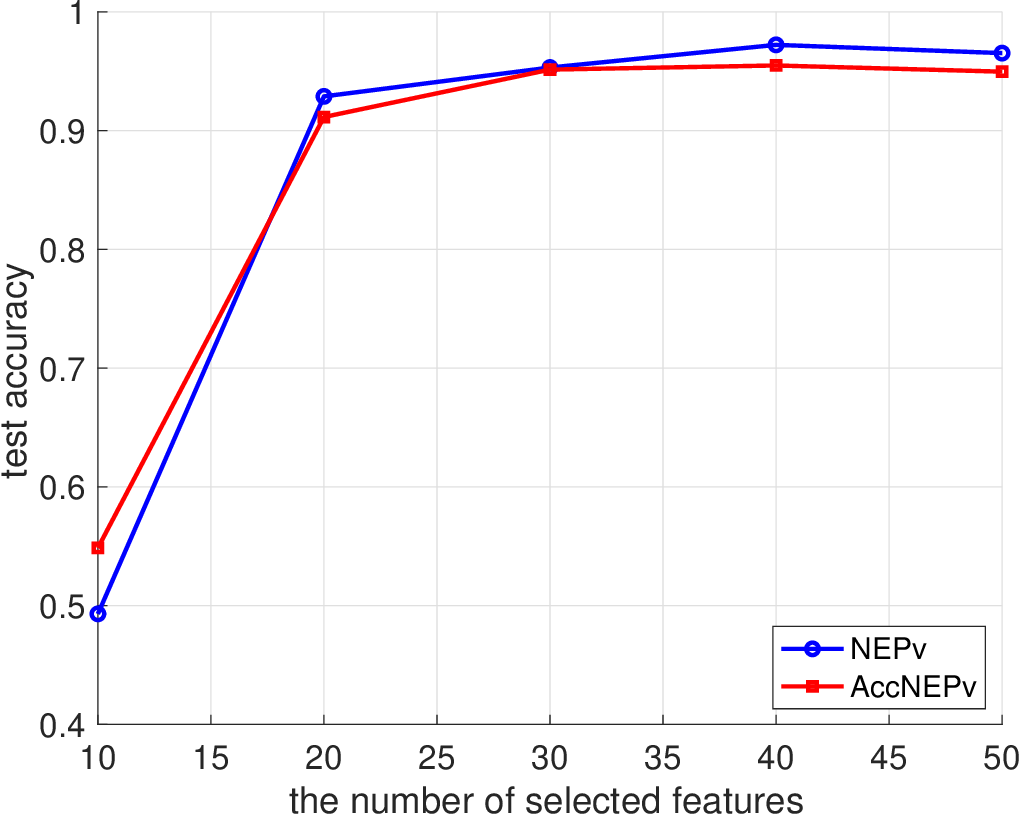} &
\includegraphics[width=0.28 \textwidth]{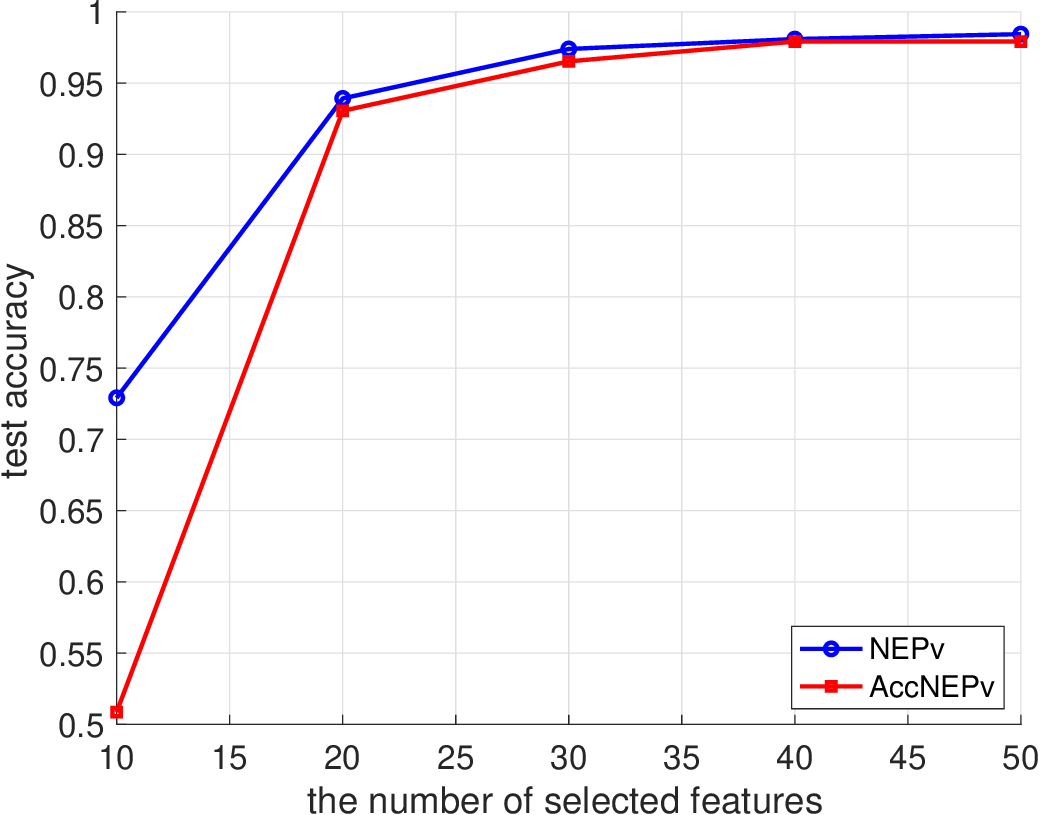} \\
\hline
\end{tabular}
\caption{\footnotesize Comparison between NEPv and AccNEPv in terms of objective value during the NEPv optimization processes and
classification accuracy on datasets COIL20-$t$ for $t\in\{1, 2, 3, 4, 5, 6\}$.} \label{fig:our-fast}
\end{figure}


For comparing NEPv and AccNEPv on the six simulated datasets with input feature dimensions
from $2024$ to $7024$ as $t$ varies in $\{1, 2, 3, 4, 5, 6\}$,
we  look at how the objective value changes as a function of elapsed CPU time as the optimization process proceeds.
This is plotted in Figure~\ref{fig:our-fast}.
First, AccNEPv not only achieves the objective value similar to NEPv, but also run much faster than NEPv. Second, from Table~\ref{tab:total-time} as the number of input features increases, the speed-up of
AccNEPv over NEPv increases too. Third, AccNEPv performs competitively to   NEPv in terms of classification accuracy.

\renewcommand{\arraystretch}{1.2}
\begin{table}
	\caption{CPU time (in seconds) on the simulated datasets COIL20-$t$} \label{tab:total-time}
	\centering
	\begin{tabular}{|c|c|c|c|c|c|c|}
		\hline
		method & $t=1$ &$t=2$ &$t=3$ &$t=4$ &$t=5$ & $t=6$\\
		\hline
		NEPv & $8.2 \cdot 10^3$ & $1.7 \cdot 10^4$ & $3.4\cdot 10^4$ & $6.2\cdot 10^4$ & $1.0\cdot 10^5$ & $1.5 \cdot 10^5$\\
		AccNEPv & $4.2\cdot 10^3$ &$5.6\cdot 10^3$ &$5.7\cdot 10^3$ &$5.9\cdot 10^3$ &$5.5\cdot 10^3$ &$4.6\cdot 10^3$\\
		\hline
	\end{tabular}
\end{table}

These experiments on the simulated datasets COIL20-$t$ reveal that the plain NEPv can be impractical for datasets with more than $10^4$ features. Unfortunately, that kind of sizes is quite common when it comes to text classification, for which the number of features is the number of words and
often there are  more than $10^4$ words, as it can be seen from
Table~\ref{tab:text-data}, where the statistics of four publicly available text
datasets\footnote{http://www.cad.zju.edu.cn/home/dengcai/Data/TextData.html} is displayed.
In what follows, we  report our experimental results on these four datasets by AccNEPv
only
because the numbers of words in the datasets are too large for NEPv to finish its computations within
reasonable amount of time (like in a day).
\begin{table}[h]
	\caption{Statistics of the four text datasets.} \label{tab:text-data}
	\centering
	\begin{tabular}{l|cccc}
		\hline
		 & RCV1\_4Class & Reuters21578 & TDT2 & 20NewsHome \\\hline
		number of documents ($p$) & 9625 & 8293 & 9394 & 18774\\
		 number of words  ($n$) & 29992 & 18933 & 36771 & 61188\\
		 number of class labels  $(k)$ & 4 & 65 & 30 & 20 \\              \hline
	\end{tabular}
\end{table}
Figure~\ref{fig:our-fast-text} and Table~\ref{tab:total-time-text} summarize our
experimental results by AccNEPv.
From Figure~\ref{fig:our-fast-text}, it can be seen that AccNEPv converges very fast and yields good classification accuracy.
AccNEPv takes a few hours to finish the experiments as indicated by Table~\ref{tab:total-time-text}.
For example, NEPv takes about 12 hours for 20NewsHome, which has more than 60 thousands of features and 18 thousands of  documents. These experimental results demonstrate that AccNEPv is effective and works very well
for large-scale high-dimensional real-world datasets.

\begin{table}[h]
	\caption{ CPU time (in hours) by AccNEPv on the four text datasets.} \label{tab:total-time-text}
	\centering
	\begin{tabular}{lcccc}
		\hline
		dataset & RCV1\_4Class & Reuters21578 & TDT2 & 20NewsHome \\\hline
		AccNEPv & $9.98$ & $10.97$ & $11.19$ & $12.46$ \\
		\hline
	\end{tabular}
\end{table}

\begin{figure}[!ht]
\centering
\begin{tabular}{c|c|c|c}
		\hline
{\footnotesize RCV1\_4Class} & {\footnotesize Reuters21578} & {\footnotesize TDT2} & {\footnotesize 20NewsHome}\\
\includegraphics[width=.22 \textwidth]{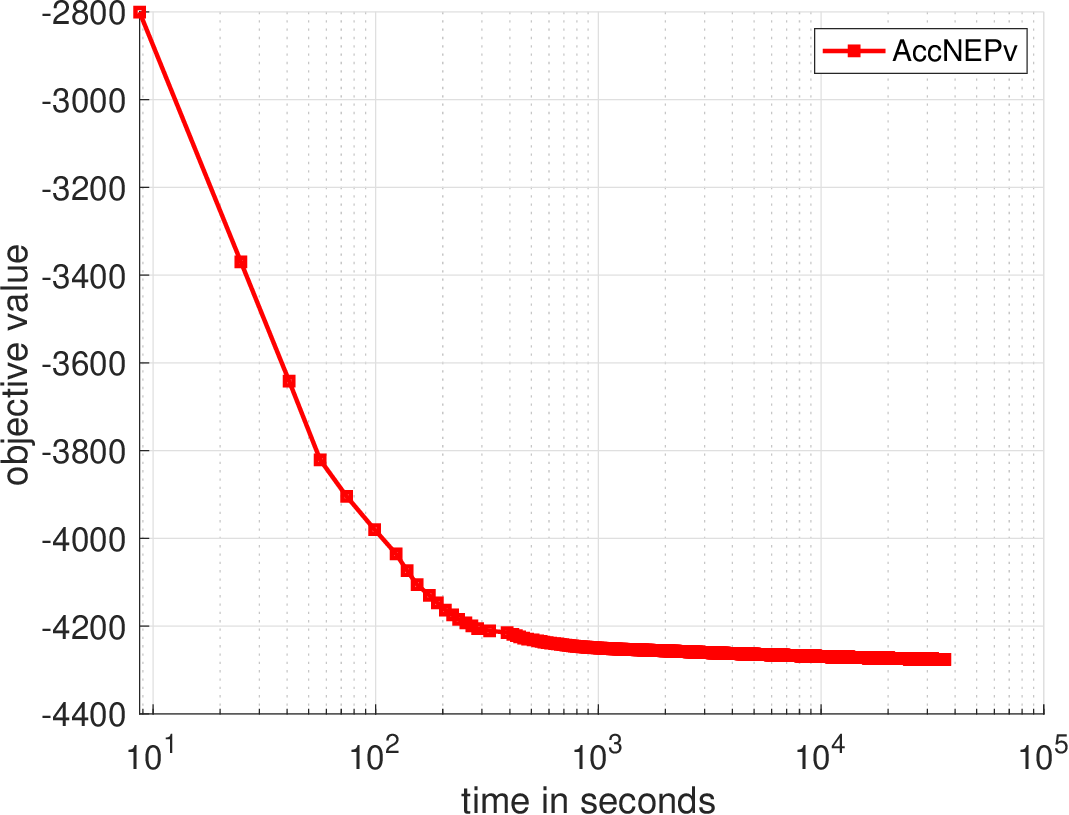}
    & \includegraphics[width=.22 \textwidth]{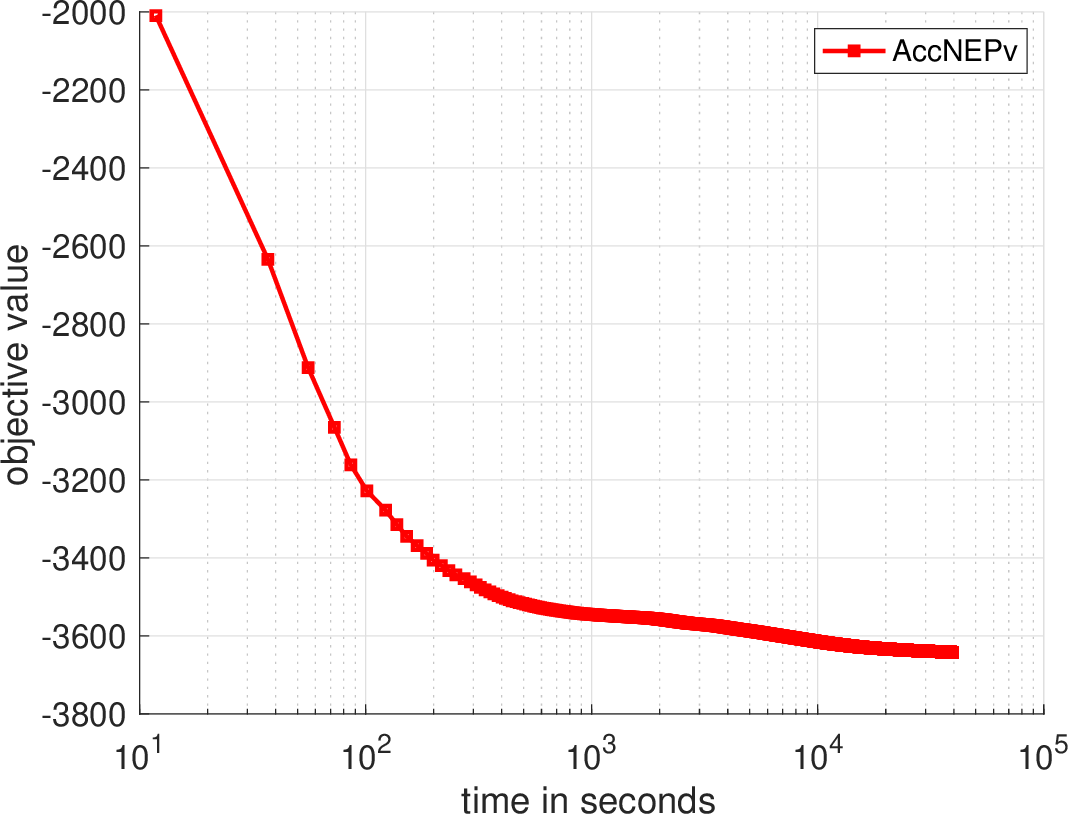}  
    & \includegraphics[width=.22 \textwidth]{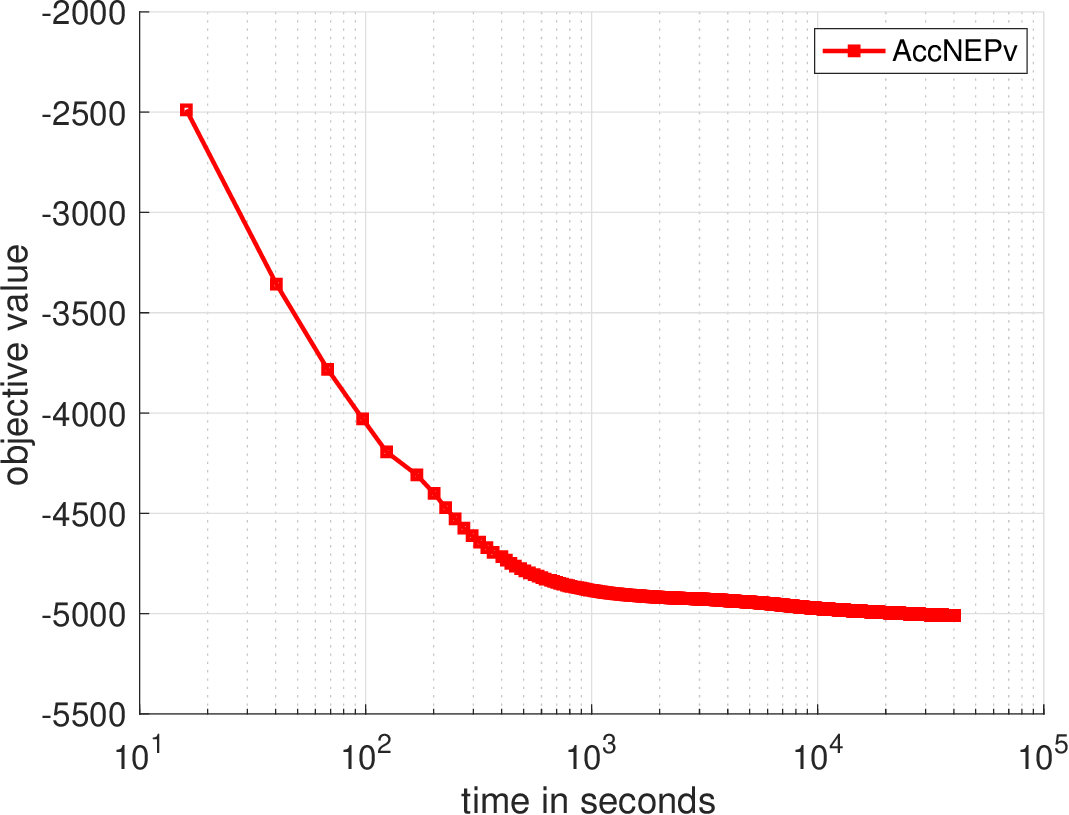}
    & \includegraphics[width=.22 \textwidth]{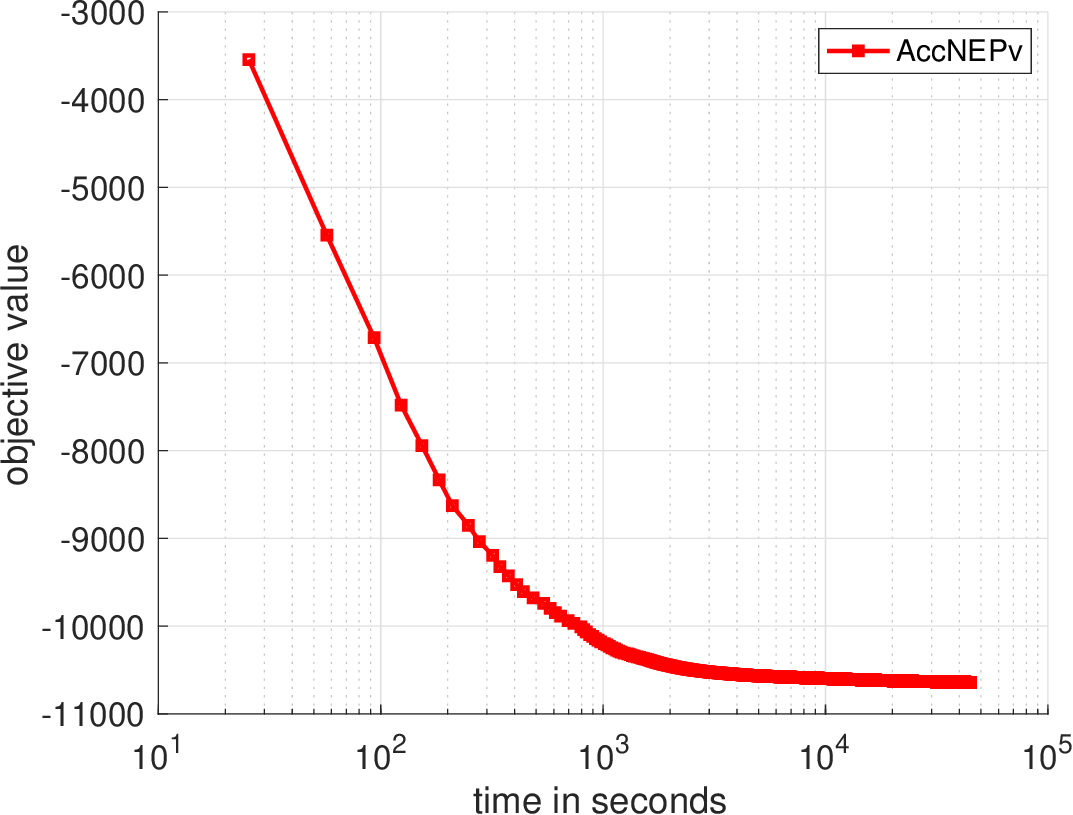}  \\
\includegraphics[width=.22 \textwidth]{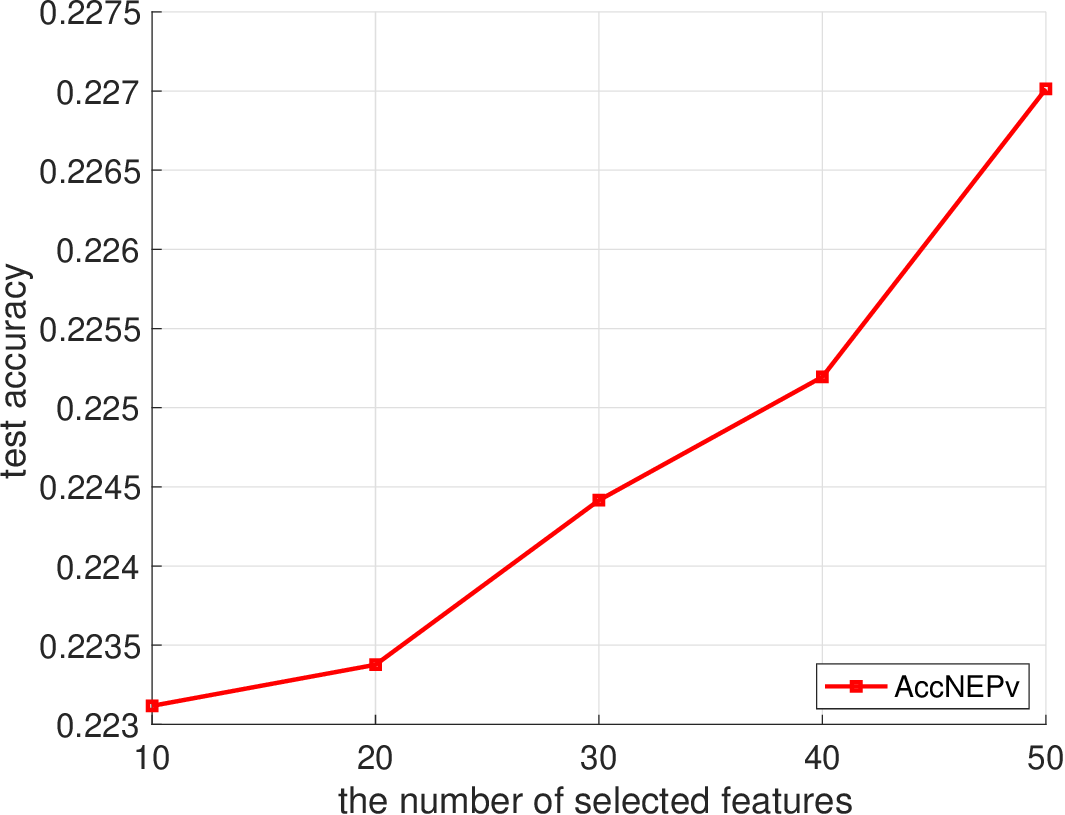}
	& \includegraphics[width=.22 \textwidth]{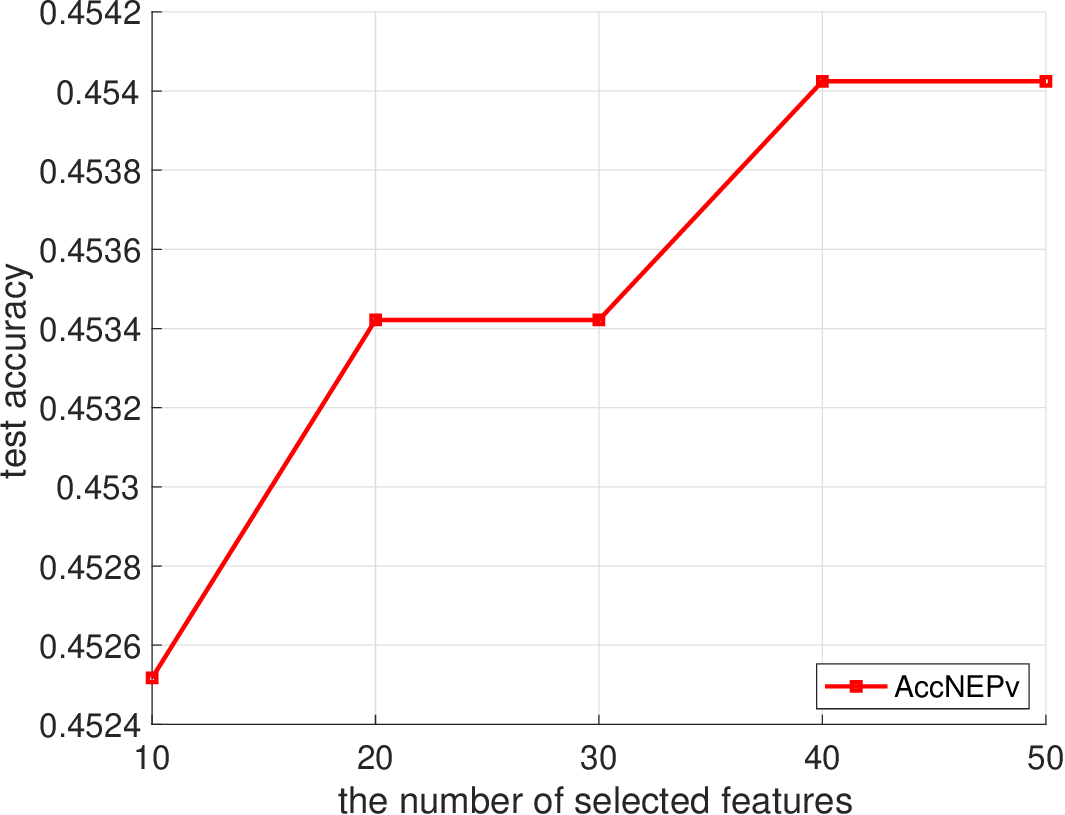}
    & \includegraphics[width=.22 \textwidth]{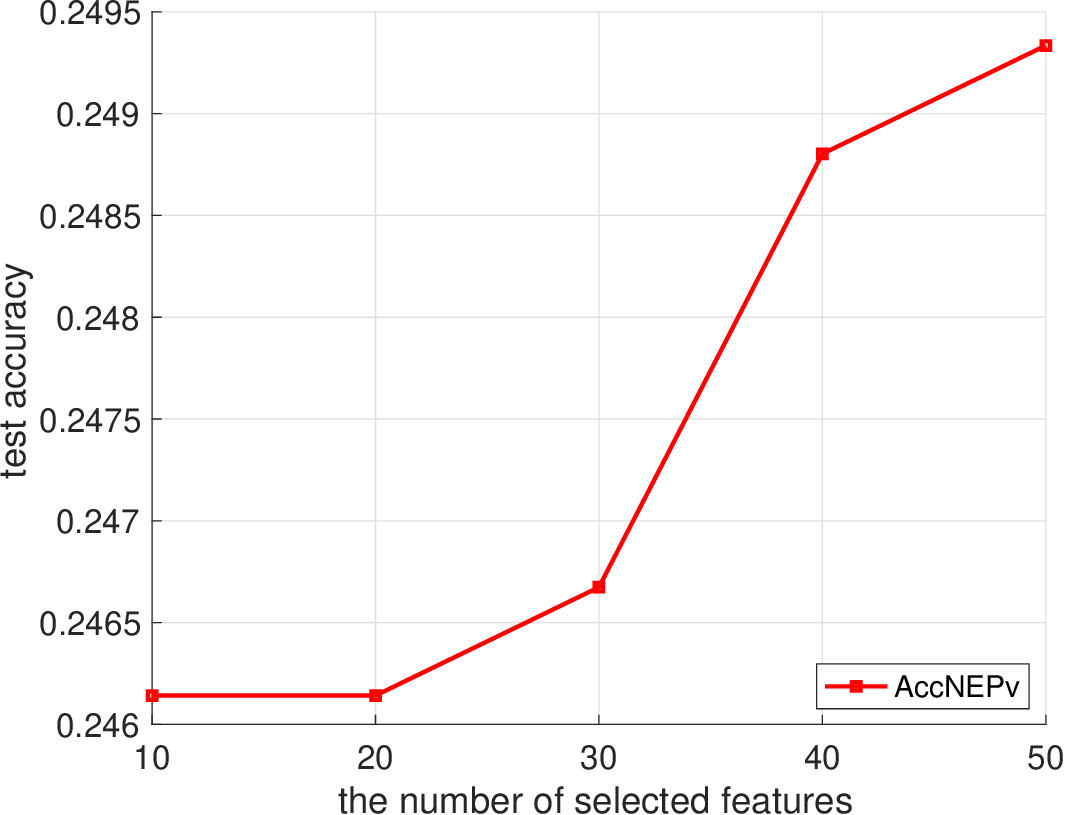}
    & \includegraphics[width=.22 \textwidth]{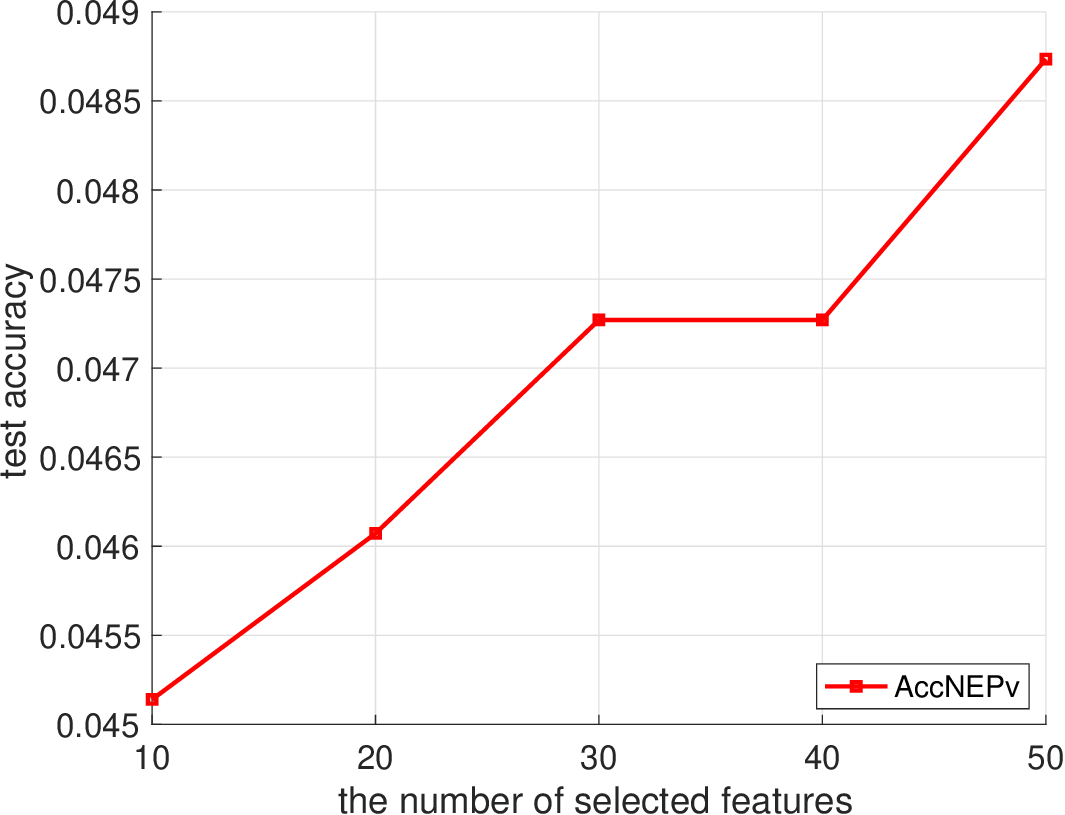}
%
	\end{tabular}
\caption{\footnotesize Experimental results by AccNEPv on the four text datasets in \Cref{tab:text-data} in terms of both objective value during optimization iterations and  classification accuracy.} \label{fig:our-fast-text}
\end{figure}


\section{Conclusion}\label{sec:concl}
In this paper, motivated by orthogonal canonical correlation analysis (OCCA) \cite{zhwb:2022},  we
propose a novel supervised feature selection model that combines OCCA with the $(2,1)$-norm regularization, which we call
OCCA21.
It is explained that the new model has a close mathematical connection to
the supervised feature selection model OLSR-OS21
(orthogonal least squares regression with optimal scaling and with the $(2,1)$-norm regularization)
\cite{Zhang2018}.
It turns out that OCCA21 falls into the class of optimization problems on the Stiefel manifold,  because of \eqref{eq:L21->trace}, for which
the NEPv approach \cite{li:2024} is theoretically sound and numerically efficient.
Through numerical experiments, we demonstrate the performances of our eventual feature selection method, OCCA-FS,
in comparison with PEB-FS and other state-of-the-art feature selection methods in the literature, in the areas
of both numerical optimization and feature selection. It is concluded that OCCA-FS computes much better optimizers than PEB-FS, produces superior classification performance to PEB-FS
and often OCCA-FS comes out on the top among all feature selection methods in comparison.

It is noted that in \cite{Zhang2018} OLSR-OS21 is solved alternatingly between its two optimizing variables.
Unfortunately, one of the two alternating updates is not done correctly, invalidating the major convergence
claim  in \cite{Zhang2018}. That incorrectness inevitably affects the classification performance
of PEB-FS, the eventual feature selection method of \cite{Zhang2018}, as confirmed by our extensive numerical experiments.

{\small
\def\noopsort#1{}\def\l{\char32l}\def\v#1{{\accent20 #1}}
  \let\^^_=\v\def\hbk{hardback}\def\pbk{paperback}

}

\end{document}